\documentclass[oneside,a4paper]{amsart}

\usepackage[T1]{fontenc}

\usepackage{amsmath,amssymb,amscd,mathrsfs,epic,empheq}

\usepackage[colorlinks=true,hyperindex, linkcolor=magenta, pagebackref=false, citecolor=cyan,pdfpagelabels]{hyperref}
\usepackage[all]{xy}
\usepackage{mathabx}
\usepackage{placeins} 
\usepackage[usenames, dvipsnames]{xcolor} 
\usepackage{euscript} 
\usepackage{bbm} 
\usepackage{enumitem} 
\usepackage[nameinlink,capitalise,noabbrev]{cleveref} 
\usepackage{todonotes}

\hypersetup{
    colorlinks=true,
    linkcolor=PineGreen,
    filecolor=magenta,      
    citecolor=cyan,
}

\usepackage{euscript}
\usepackage{graphicx}
\usepackage{stmaryrd}

\usepackage[margin=1in]{geometry}

\usepackage{tikz}
\usetikzlibrary{arrows,backgrounds,
    decorations.pathreplacing,
    decorations.pathmorphing
}
\usetikzlibrary{matrix,arrows}

\numberwithin{equation}{subsection}

\newtheorem{Proposition}{Proposition}[section]
\newtheorem{Lemma}[Proposition]{Lemma}
\newtheorem{Theorem}[Proposition]{Theorem}

\newtheorem{Corollary}[Proposition]{Corollary}

\theoremstyle{definition}
\newtheorem{Definition}[Proposition]{Definition}
\newtheorem{Construction}[Proposition]{Construction}
\newtheorem{Remark}[Proposition]{Remark}
\newtheorem{Convention}[Proposition]{Convention}
\newtheorem{Notation}[Proposition]{Notation}
\newtheorem{Example}[Proposition]{Example}

\newtheorem{Warning}[Proposition]{Warning}

\newcommand{\euXym}{\ensuremath{\euXymatrix@1}}

\newcommand{\C}{\ensuremath{\EuScript{C}}}
\newcommand{\euP}{\ensuremath{\EuScript{P}}}

\newcommand{\dcat}{\ensuremath{\mathcal{D}}}

\newcommand{\euY}{\ensuremath{\EuScript{Y}}}
\newcommand{\euX}{\ensuremath{\EuScript{X}}}

\newcommand{\euZ}{\ensuremath{\EuScript{Z}}}
\newcommand{\euW}{\ensuremath{\EuScript{W}}}

\newcommand{\eF}{\ensuremath{\EuScript{F}}}

\newcommand{\Ccat}{\mathcal{C}}
\newcommand{\Dcat}{\mathcal{D}}
\newcommand{\Ncat}{\mathcal{N}}
\newcommand{\eupsilon}{\ensuremath{\EuScript{E}}}

\newcommand{\deloop}{\ensuremath{\EuScript{B}}}
\newcommand{\twistedcurves}{\ensuremath{\mathfrak{M}^{\text{tw}}}}

\newcommand{\vir}{\ensuremath { \text{vir}}}

\def\hom{{\operatorname{Hom}}}

\def\Ext{\operatorname{Ext}}

\def\op{\operatorname{op}}

\def\p{\operatorname{p}}

\def\Mod{\operatorname{Mod}}

\def\dim{\operatorname{dim}}

\def\spec{\operatorname{Spec}}

\def\coh{\operatorname{Coh}}
\def\qcoh{\operatorname{QCoh}}

\def\Fun{\operatorname{Fun}}

\def\fib{\operatorname{fib}}
\def\map{\operatorname{map}}

\def\calg{\operatorname{CAlg}}

\def\sym{\operatorname{Sym}}
\def\spaces{\EuScript{S}}

\def\sheaves{\operatorname{Sh}}

\def\presheaves{\operatorname{PSh}}

\def\affines{\operatorname{\EuScript{A}ff}}
\def\stacks{\operatorname{\EuScript{S}tk}}
\def\pairs{\operatorname{\EuScript{P}air}}
\def\prestacks{\operatorname{\EuScript{P}r\EuScript{S}tk}}
\def\artinstacks{\operatorname{\EuScript{A}rt}}
\def\relativeartinstacks{\operatorname{\EuScript{R}el\EuScript{A}rt}}
\def\ab{\operatorname{Ab}}

\def\fib{\operatorname{fib}}

\def\chow{\operatorname{CH}}

\newcommand{\triplerightarrow}{%
\tikz[minimum height=0ex]
  \path[->]
   node (a)            {}
   node (b) at (1em,0) {}
  (a.north)  edge (b.north)
  (a.center) edge (b.center)
  (a.south)  edge (b.south);%
}

\makeatletter
  \def\subsection{\@startsection{subsection}{1}%
  \z@{.7\linespacing\@plus\linespacing}{.5\linespacing}%
  {\normalfont\bfseries\centering}}
\makeatother

\let\oldtocsection=\tocsection
\let\oldtocsubsection=\tocsubsection
\let\oldtocsubsubsection=\tocsubsubsection
\renewcommand{\tocsection}[2]{\hspace{0em}\oldtocsection{#1}{#2}}
\renewcommand{\tocsubsection}[2]{\hspace{1em}\oldtocsubsection{#1}{#2}}
\renewcommand{\tocsubsubsection}[2]{\hspace{2em}\oldtocsubsubsection{#1}{#2}}

\calclayout

\begin{document}
\title[The Intrinsic Normal Cone for Artin Stacks]{The Intrinsic Normal Cone for Artin Stacks}
\author[Dhyan Aranha, Piotr Pstr\k{a}gowski]{Dhyan Aranha, Piotr Pstr\k{a}gowski}
\email{dhyan.aranha@gmail.com, pstragowski.piotr@gmail.com}

\begin{abstract}
We extend the construction of the normal cone of a closed embedding of schemes to any locally of finite type morphism of higher Artin stacks and show that in the Deligne-Mumford case our construction recovers the relative intrinsic normal cone of Behrend and Fantechi. We characterize our extension as the unique one satisfying a short list of axioms, and use it to construct the deformation to the normal cone. As an application of our methods, we associate to any morphism of Artin stacks equipped with a choice of a global perfect obstruction theory a relative virtual fundamental class in the Chow group of Kresch.
\end{abstract}

\maketitle 

\tableofcontents

\section{Introduction} 

Moduli spaces often have an expected or so called "virtual" dimension at each point, which is a lower bound for the actual dimension. An important example due to Kontsevich \cite{Kont} is given by the moduli stack $\overline{\mathcal{M}}_{g, n}(V, \beta)$ of stable maps of degree $ \beta \in H_2(V)$ from $n$-marked prestable curves of genus $g$ into a smooth projective variety $V$. This is a proper Deligne-Mumford stack whose actual dimension at a given point $(C, f)$ will in general be larger than it's virtual dimension, given by
\begin{equation*}
3g-3+n + \chi(C, f^*T_V) = (1-g)(\dim V - 3) - \beta(\omega_V) +n.
\end{equation*}
The moduli of stable curves is used to define Gromov-Witten invariants of $V$, and one of the key ingredients is to be able to construct a virtual fundamental class
\begin{equation*}
[\overline{\mathcal{M}}_{g, n}(V, \beta)]^{\text{vir}} \in A_{(1-g)(\dim V - 3) - \beta(\omega_V) +n} (\overline{\mathcal{M}}_{g, n}(V, \beta)).
\end{equation*}
of the expected dimension. There is a general procedure for constructing such classes whenever the moduli space in question is a Deligne-Mumford stack equipped with a choice of a perfect obstruction theory due to Behrend and Fantechi \cite{behrend_fantechi_intrinsic_normal_cone}.

The construction of the above virtual fundamental classes allowed Behrend to construct Gromov-Witten invariants for arbitrary smooth projective varieties and arbritrary genus satisfying the axioms of Kontsevich and Manin \cite{Beh}, \cite{KontManin}. Virtual fundamental classes are also are also foundational objects of several other enumerative theories, such as those of Donaldson-Thomas, Pandharipande-Thomas, and Vafa-Witten invariants \cite{DT}, \cite{PT}, \cite{TT}.

In this paper, we extend the methods of Behrend and Fantechi to the setting of higher Artin stacks. Reducing to the classical case, where we have access to the Chow groups of Kresch, we are then able to construct a virtual fundamental class in a wide context.

\begin{Theorem}
\label{Theorem:existence_of_virtual_class_intro} 
Let $\euX \rightarrow \euY$ be a morphism of finite type Artin stacks. Suppose that  $\euY$ is of pure dimension $r$ and that we have a perfect obstruction theory $\eupsilon \rightarrow L_{\euX / \euY}$ which admits a global resolution. Then, there is a well-defined virtual fundamental class $[\euX \rightarrow \euY, \eupsilon]^{\vir} \in \chow_{r + \chi(\eupsilon)}(\euX)$ in the Chow group of $\euX$. 
\end{Theorem}

To mention a couple of examples to which \cref{Theorem:existence_of_virtual_class_intro} applies, we have
\begin{enumerate}
\item  the moduli of twisted stable maps whose target is an Artin stack, see \cref{Example:twisted_stable_maps} and 
\item the $0$-truncation of any quasi-smooth morphism of derived Artin stacks, in particular the moduli spaces arising in Donaldson-Thomas theory, see \cref{Example:quasi-smooth_derived_stacks}.
\end{enumerate}
The theory of twisted stable maps is of particular importance, as it allows one to construct generalizations of Gromov-Witten invariants; this will be explored in forthcoming work. 

To obtain the needed virtual fundamental class, Behrend and Fantechi associate to any morphism $\euX \rightarrow \euY$ of Deligne-Mumford type the \emph{intrinsic normal cone} $\Ccat_{\euX} \euY$, which is a closed substack of the normal sheaf. Informally, the virtual fundamental class is then obtained by intersecting the class of the normal cone with the zero section of abelian cone of the chosen perfect obstruction theory, mirroring a classical construction of Fulton \cite{fulton2013intersection}. 

In general, the intrinsic normal cone of a Deligne-Mumford stack is only Artin rather than Deligne-Mumford, and likewise it turns out that the natural definition of the intrinsic normal cone $\Ccat_{\euX}$ where $\euX$ is Artin forces the cone to be a \emph{higher} Artin stack; that is, an \'{e}tale sheaf on the site of schemes valued in the $\infty$-category of spaces rather than in groupoids. Thus, to obtain the correct generalization we are forced to work in the setting of higher algebraic stacks. 

Since we work with $\infty$-categories, it is often easier to uniquely characterize a given construction rather than to write it down directly. This is exactly what we do, and so our work offers some conceptual clarification even in the classical context. 

Let us say that a \emph{relative higher Artin stack} $\euX \rightarrow \euY$ is a locally of finite type morphism of higher Artin stacks. We denote the $\infty$-category of relative higher Artin stacks with morphisms given by commutative squares by $\relativeartinstacks$. We will say a morphism of relative higher Artin stacks 

\begin{center}
	\begin{tikzpicture}
		\node (TL) at (0, 1.3) {$ \euX^{\prime} $};
		\node (TR) at (1.5, 1.3) {$ \euY^{\prime} $};
		\node (BL) at (0, 0) {$ \euX $};
		\node (BR) at (1.5, 0) {$ \euY $};
		
		\draw [->] (TL) -- (TR);
		\draw [->] (TL) -- (BL);
		\draw [->] (TR) -- (BR);
		\draw [->] (BL) -- (BR);
		\end{tikzpicture}
\end{center}
is \emph{smooth} if both vertical arrows are smooth and \emph{surjective} if both vertical arrows are surjective. 

If $\euX \rightarrow \euY$ is a relative Artin stack, then its \emph{normal sheaf} $\Ncat_{\euX} \euY := \mathbb{V}_{\euX}(L_{\euX / \euY}[-1])$ is defined as the abelian cone associated to the shift of the cotangent complex. Our first result provides a unique characterization of this construction.

\begin{Theorem}[\ref{Theorem:unique_characterization_of_a_normal_sheaf_of_morphisms_of_artin_stacks}]
\label{Theorem:introduction_characterization_of_normal_sheaf}
The normal sheaf functor $\Ncat: \relativeartinstacks \rightarrow \artinstacks$ is characterized uniquely by the following properties: 

\begin{enumerate}
\item If $U \hookrightarrow V$ is a closed embedding of schemes, then $\Ncat_{U}V$ coincides with the normal sheaf in the classical sense, that is, $\Ncat_{U}V \simeq \mathbb{V}_{U}(I / I^{2})$, where $I$ is the ideal sheaf
\item $\Ncat$ preserves coproducts.
\item $\Ncat$ preserves smooth and smoothly surjective maps.
\item $\Ncat$ commutes with pullbacks along smooth morphisms.
\end{enumerate}
\end{Theorem}
It is not difficult to see that any functor satisfying the above properties is a cosheaf on $\relativeartinstacks$ with respect to the topology determined by smoothly surjective maps, and so \cref{Theorem:introduction_characterization_of_normal_sheaf} is strongly related to the flat descent for the cotangent complex \cite{bhatt2012completions}. 

Since we work only with discrete rings, the abelian cone associated to a quasi-coherent sheaf depends only on its coconnective part, which one can in fact recover from the abelian cone. Thus, \cref{Theorem:introduction_characterization_of_normal_sheaf} can be interpreted as saying that the "naive" cotangent complex $(L_{\euX / \euY})_{\leq 1}$ is already determined by its behaviour on closed embeddings of schemes.

\begin{Theorem}[\ref{Theorem:there_exists_a_unique_normal_cone_functor_subject_to_properties}, \ref{Theorem:normal_cone_is_a_closed_substack_of_normal_sheaf_and_has_cartesian_property_for_smooth_morphisms}]

\label{Theorem:introduction_there_exists_a_unique_normal_cone_functor_subject_to_properties}
There exists a unique functor $\ \Ccat: \relativeartinstacks \rightarrow \artinstacks$, called the \emph{normal cone}, which satisfies the following properties:
\begin{enumerate}
\item If $U \hookrightarrow V$ is a closed embedding of schemes, then $\Ccat_{U}V$ coincides with the classical normal cone, that is, $\Ccat_{U}V \simeq \spec_{U}(\bigoplus I^{k} / I^{k+1})$, where $I$ is the ideal sheaf.
\item $\Ccat$ preserves coproducts.
\item $\Ccat$ preserves smooth and smoothly surjective maps.
\item $\Ccat$ commutes with pullbacks along smooth morphisms of relative Artin stacks.
\end{enumerate}
Moreover, there is a natural map $\Ccat_{\euX} \euY \hookrightarrow \Ncat_{\euX} \euY$ which is a closed embedding for an arbitrary relative higher Artin stack $\euX \rightarrow \euY$.
\end{Theorem}
Note that \cref{Theorem:introduction_there_exists_a_unique_normal_cone_functor_subject_to_properties} is qualitatively different from our axiomatization of the normal sheaf, where we have the construction using the cotangent complex, as part of the statement is that the needed functor exists. Rather, \cref{Theorem:introduction_characterization_of_normal_sheaf} should be thought of as suggesting that the above set of axioms on the normal cone is the right one. This is further evidenced by the following comparison with the construction of Behrend and Fantechi.

\begin{Theorem}[\ref{Theorem:comparison_with_Behrend_Fantechi_normal_cone}]
\label{Theorem:introduction_comparison_of_normal_cone_with_intrinsic_normal_cone_of_behrend_fantechi}
Let $\euX \rightarrow \euY$ be relatively Deligne-Mumford morphism of Artin stacks of finite type. Then, the normal cone $\Ccat_{\euX} \euY$ coincides with the relative intrinsic normal cone of Behrend and Fantechi.
\end{Theorem}

The proofs of \cref{Theorem:introduction_characterization_of_normal_sheaf} and \cref{Theorem:introduction_there_exists_a_unique_normal_cone_functor_subject_to_properties} use what we call \emph{adapted cosheaves}. Roughly, a functor $F: \Ccat \rightarrow \mathcal{H}$ from an $\infty$-site into an $\infty$-topos is an adapted cosheaf it it satisfies descent and preserves pullbacks along a distinguished class of geometric morphisms which contains all coverings. Our main result shows that an adapted cosheaves are stable under left Kan extension from a generating subcategory.

In our case, $\Ccat$ is the $\infty$-category of relative higher Artin stacks, the distinguished class of maps is given by smooth morphisms, and the generating subcategory is the category of closed embeddings of schemes. This method is very general, and allows one to construct other functors related to the normal cone, for example the deformation space. 

\begin{Theorem}[\ref{Theorem:existence_and_properties_of_deformation_space}]
\label{Theorem:introduction_existence_of_deformation_space}
For any relative higher Artin stack $\euX \rightarrow \euY$ there exists a higher Artin stack $M^\circ_{\euX} \euY$ which fits into a commutative diagram

 \begin{center}
	\begin{tikzpicture}
		\node (TL) at (-1 , 1) {$ \euX \times \mathbb{P}^{1} $};
		\node (TR) at (1, 1) {$ M^\circ_{\euX} \euY  $};
		\node (B) at (0, 0){$ \mathbb{P}^{1} $};
		
		\draw[right hook->] (TL) -- (TR);
		\draw[->] (TL) -- (B);
		\draw[->] (TR) -- (B);
	\end{tikzpicture}
\end{center}
where both vertical arrows are flat and such that

\begin{enumerate}
\item over $\mathbb{A}^{1} \simeq \mathbb{P}^{1} - \{ \infty \}$, the horizontal arrow is equivalent to $\euX \times \mathbb{A}^{1} \hookrightarrow \euY \times \mathbb{A}^{1}$ and
\item over $\{ \infty \}$, the horizontal arrow is equivalent to $\euX \hookrightarrow \Ccat_{\euX}\euY$.
\end{enumerate}
\end{Theorem} 

We point out that the assumptions of being Artin in the classical sense and of the existence of a global resolution appearing in \cref{Theorem:existence_of_virtual_class_intro} stem only from the fact that we are not aware of a theory of \emph{integral} Chow groups for higher Artin stacks which has the needed properties. 

In the rational case, the needed Chow groups were constructed in great generality by Khan, see \cite{Khan}, so that our methods yield the needed virtual fundamental class for any choice of perfect obstruction theory. Similarly, these virtual fundamental classes always exist in $K$-theory, see \cref{Remark:virtual_fundamental_class_in_k_theory}. 

Lastly, the restriction to morphisms locally of finite type comes from the fact that any such morphism of higher Artin stacks admits a smooth surjection from a closed embedding of schemes. We believe it is likely that the normal cone satisfies the analogues of the axioms of \cref{Theorem:introduction_there_exists_a_unique_normal_cone_functor_subject_to_properties} with the class of smooth maps replaced by that of flat maps, as that is the case for the normal sheaf. If that was the case, the locally of finite type assumption could be removed throughout.

\subsection{Notation and conventions}

In the sequel we will use the term \emph{Artin stack} to refer to what we called a higher Artin stack in the introduction, see \cref{Definition:Artin_stacks}. Under this convention, classical Artin stacks correspond to what we call \emph{$1$-Artin} stacks. 

To tackle coherence difficulties inherent in working with functors valued in spaces, we will use the framework of \emph{$\infty$-categories}, as developed by Joyal and Lurie. The standard reference is \cite{higher_topos_theory}. 

Throughout this paper, we will be working over a fixed field $k$. Note that even though we work with $\infty$-categories, we will be indexing our stacks using the category of \emph{discrete} commutative $k$-algebras, which we denote by $\calg_{k}$. Derived analogues of commutative rings will appear only indirectly. 

\subsection{Relation to other works}

As mentioned in the introduction, Adeel Khan has constructed Chow groups of derived higher Artin stacks rationally, using motivic homotopy theory \cite{Khan}. In the same work, Khan associates a canonical virtual fundamental class to any quasi-smooth morphism of derived Artin stacks using the axiomatics of bivariant homology theories. It follows from his work that in this context, the image of Khan's fundamental class in the Chow group of the classical truncation agrees with the one constructed in this paper, where the perfect obstruction theory is the one induced by the chosen quasi-smooth derived enhancement.

\subsection{Acknowledgements}

We would like to thank Anthony Blanc and especially Barbara Fantechi for many conversations and support during the gestation period of this work. The second author would like to thank the first author for inviting him to work on this project. We would like to thank Jacob Lurie for sharing with us an early draft of his chapter on derived stacks in \cite{spectral_algebraic_geometry}.

Lastly, after putting this work on the Arxiv, Adeel Khan alerted us about \cite{Khan} which had been put up some days before ours, see above. We would like to thank Adeel for very graciously inviting us to Regensburg to speak about our work as well as for the time he took to explain his work to us.

\section{Preliminaries}

The goal of this section is to introduce the reader to the language we will use throughout this paper. Our aim is to simply collect all the background material that will be needed to understand the main results. To this end, we have taken a minimalist approach, providing references to more thorough treatments in the literature. 

\subsection{Higher Artin stacks}

In algebraic geometry, especially in moduli theory, one often cares about functors which are not naturally valued in sets, but rather in groupoids; such functors are then given geometric interpretation through the theory of algebraic stacks. 

For some purposes - in our case, of giving a well-behaved definition of an intrinsic normal cone of an algebraic stack - even the category of functors valued in groupoids is not sufficient \cite{lurie_thesis_dag}, \cite{homotopical_algebraic_geometry}. One is then naturally led to consider functors valued in \emph{spaces}; the geometric interpretation of such functors is given by the theory of \emph{higher} algebraic stacks, first studied by Simpson \cite{Simp}.

\begin{Remark}
For the purpose of this paper, the word \emph{space} will be synonymous with the word \emph{$\infty$-groupoid}; that is, an $\infty$-category where all of the morphisms are invertible \cite{higher_topos_theory}[1.2.5.1]. 

The need for a suitable theory of higher groupoids needed to define higher algebraic stacks was recognized by Simpson \cite{Simp}. More precisely, to define $n$-stacks one needs $n$-groupoids, which we can identify with $n$-truncated $\infty$-groupoids, or what we will call in this paper \emph{$n$-truncated spaces}. 
\end{Remark}

\begin{Definition}
The category $\affines$ of \emph{affine $k$-schemes} is the opposite of the category $\calg_{k}$ of discrete $k$-algebras. We will consider $\affines$ as a Grothendieck site with respect to the \'{e}tale topology. 
\end{Definition}
A scheme can be identified with a particular sheaf of sets over $\affines$; similarly, an algebraic stack over $k$ can be identified with an appropriate sheaf of groupoids. As explained above, when working with higher stacks, we instead allow sheaves valued in spaces. 

\begin{Definition}
A \emph{prestack} $X$ is a presheaf over $\affines$ valued in the $\infty$-category $\spaces$ of spaces; that is, it is a functor $X: \affines^{op} \rightarrow \spaces$. We say a prestack $X$ is a \emph{stack} if it is a sheaf with respect to the \'{e}tale topology. We denote the $\infty$-categories of (pre)stacks by $\prestacks$ and $\stacks$. 
\end{Definition}
Note that any set can be considered as a discrete space, so that any presheaf of sets gives rise to a prestack as above. Moreover, in this case the $\infty$-categorical sheaf condition reduces to the usual one. 

\begin{Remark}
Recall that if $X$ is a presheaf of sets on the site of affine schemes, then $X$ is a sheaf if and only if it preserves products and for any \'{e}tale surjection $U \rightarrow V$ of  affine $k$-schemes, the diagram $X(V) \rightarrow X(U) \rightrightarrows X(U \times_{V} U)$ is a limit. 

In the case of a presheaf of spaces, to only consider the two-fold intersections is not enough, and one instead requires that the whole diagram 

\begin{center}
$X(V) \rightarrow X(U) \rightrightarrows X(U \times _{V} U) \triplerightarrow X(U \times _{V} U \times _{V} U) \ldots$
\end{center}
induced by the \v{C}ech nerve of $U \rightarrow V$ is a limit diagram of spaces.
\end{Remark}

Note that in our definition of a stack, we allow sheaves of spaces, but we still index them by the classical category of discrete $k$-algebras, rather than a derived variant. Thus, we are working within the framework of classical, rather than derived, algebraic geometry. 

Nevertheless, a lot of the definitions we will work with are analogous to the ones which became standard in derived algebraic geometry. In particular, our notion of an Artin stack is analogous to the one appearing in Lurie's thesis \cite{lurie_thesis_dag}, see also \cite{Simp} and \cite{TV2} for earlier treatments.

\begin{Definition}
\label{Definition:Artin_stacks}
We define $n$-Artin stacks and smooth $n$-Artin stacks inductively as follows:

\begin{enumerate}
\item We say a morphism $ f: \euX \rightarrow \euY $ of stacks is a \emph{relative $0$-Artin stack} if for any map $ g: \spec(A) \rightarrow \euY$, the fiber product $ \spec(A) \times_\euY \euX$ is an algebraic space. 
\item We say that relative $0$-Artin stack $ f : \euX \rightarrow \euY$ is \emph{smooth} if each of the associated maps $ \spec(A) \times_\euY \euX \rightarrow \spec(A)$ is smooth as a morphism of algebraic spaces.
\item For $n > 0$, we say a morphism $ f: \euX \rightarrow \euY$ of stacks is a \emph{relative $n$-Artin stack} if for any map $ \spec(A) \rightarrow \euY$ there exits a smooth surjection $ U \rightarrow \spec(A) \times_\euY \euX$ which is a relative $(n-1)$-Artin stack, where $U $ is an algebraic space. 
\item We say that a relative $n$-Artin stack  $f : \euX \rightarrow \euY$ is \emph{smooth} if for every $\spec(A) \rightarrow \euX$ there exists a smooth surjection $ U \rightarrow \spec(A) \times_\euY \euX$ as in the previous item, such that $ U \rightarrow \spec(A)$ is a smooth morphism of schemes.
\item We say that a stack $\euX$ is an \emph{$n$-Artin stack} if it is a relative $n$-Artin stack over $ \spec(k)$ and will refer to an \emph{Artin stack} as a stack $\euX$ which is $n$-Artin for some $n$. 
\end{enumerate}
We denote the $\infty$-category  Artin stacks by $ \artinstacks$.
\end{Definition}

The next proposition is a collection of basic properties pertaining to higher Artin stacks. 

\begin{Proposition} 
\label{Proposition:properties_of_artin_morphisms}
We have that 

\begin{enumerate}
\item Any relative $n$-Artin stack is also a relative $m$-Artin stack for any $m \geq n$. 
\item A pullback of a (smooth) relative n-Artin stack is (smooth) relative n-Artin
\item Let $ \euX \overset{f}{\rightarrow} \euY \overset{g}{\rightarrow} \EuScript{Z}$ be a pair of composable morphisms. If both $f$ and $g$ are (smooth) relative $n$-Artin stacks, then so is $ g \circ f$. 
\item Suppose that $n > 0$ and that we are given morphisms $ \euX \rightarrow \euY \rightarrow \EuScript{Z} $, where $ \euX \rightarrow \euY$ is an $(n-1)$-submersion and $ \euX \rightarrow \EuScript{Z}$ is a relative $n$-Artin stack. Then $ \euY \rightarrow \EuScript{Z}$ is a relative $n$-Artin stack.
\item Let $\euX \overset{f}{\rightarrow} \euY \overset{g}{\rightarrow} \EuScript{Z}$ be a composable pair of morphisms and $ n \geq 1$. If $ g \circ f$ is a relative $(n-1)$-Artin stack and $g$ is a relative $n$-Artin stack, then $ f$ is a relative $(n-1)$-Artin stack.
\end{enumerate}
\end{Proposition}

\begin{proof}
This is \cite{lurie_thesis_dag}[5.1.4].
\end{proof}
Observe that one of the pleasant consequences of \cref{Proposition:properties_of_artin_morphisms} is that any morphism $ f: \euX \rightarrow \euY$ of $n$-Artin stacks is automatically a relative $n$-Artin stack.

\begin{Example} 
\label{Example:example_of_higher_artin_stack}
Let $\mathbb{G}_m$ denote the mulitplicative group scheme over $k$. Classically, we can form the quotient $1$-Artin stack $ \deloop \mathbb{G}_m : = [ \spec(k) // \mathbb{G}_m]$. From a homotopy-theoretic perspective, the stack $\deloop \mathbb{G}_m $ is the \'{e}tale sheafification of the presheaf
\begin{eqnarray*}
B \mathbb{G}_m : \affines^{\op} & \rightarrow &  \spaces \\
R & \mapsto  & K(R^\times, 1)
\end{eqnarray*}
where $ K(R^\times, 1)$ is the first Eilenberg-Maclane space of the abelian group of units of $R$.  Since $\mathbb{G}_m$ is abelian,  it is very natural to consider prestacks 
\begin{eqnarray*}
R & \mapsto & K(R^\times, n)
\end{eqnarray*}
for any $ n \geq 1$. We define $ \deloop^n \mathbb{G}_m$ to be \'{e}tale sheafifcation of the presheaf defined by the above formula, one can show that it is an $n$-Artin stack. 

To see this in the basic case of $n =2$, note that we have an equivalence of stacks

\begin{center}
$B^2 \mathbb{G}_m \simeq \varinjlim \ (\ldots \ \deloop \mathbb{G}_m \times \deloop \mathbb{G}_m \triplerightarrow \deloop \mathbb{G}_m  \rightrightarrows \spec(k)).$
\end{center}
We claim that the induced map $\spec(k) \rightarrow \deloop^2 \mathbb{G}_m$ is an $2$-submersion, it is clearly surjective. Furthermore, since the diagram
\begin{center}
	\begin{tikzpicture}
		\node (TL) at (0, 1.3) {$\deloop \mathbb{G}_m$};
		\node (TR) at (2, 1.3) {$\spec(k)$};
		\node (BL) at (0, 0) {$\spec(k)$};
		\node (BR) at (2, 0) {$\deloop^2 \mathbb{G}_m.$};
		
		\draw [->] (TL) -- (TR);
		\draw [->] (TL) -- (BL);
		\draw [->] (TR) -- (BR);
		\draw [->] (BL) -- (BR);
		\end{tikzpicture}
\end{center}
is a pullback diagram and the maps $ \deloop \mathbb{G}_m \rightarrow *$ are smooth relative $1$-Artin stacks on the account of their fibers being $\mathbb{G}_m$, we conclude that $\spec(k) \rightarrow \deloop^2 \mathbb{G}_m$ is a $2$-submersion. More generally, in the discussion above we could replace $\mathbb{G}_m$ with any smooth abelian group scheme.
\end{Example}

\begin{Definition}
\label{Definition:property_local_on_both_source_and_target}
Suppose that $P$ is a property of morphisms of schemes over $k$ which is local both in the source and target in the smooth topology. Then, we say a morphism $f: \euX \rightarrow \euY$ of Artin stacks \emph{has property P} if for any $\spec(A) \rightarrow \euY$ there exists a smooth surjection $S \rightarrow \spec(A) \times_{\euY} \euX$ from a scheme such that $S \rightarrow \spec(A)$ has property $P$.  
\end{Definition}

\begin{Example}
Properties local both in the source and target in the smooth topology to which we might want to apply \cref{Definition:property_local_on_both_source_and_target} include being \emph{locally of finite type}, \emph{flat} and \emph{smooth}. Note that in the smooth case the resulting notion will coincide with that of a smooth relative Artin stack of \cref{Definition:Artin_stacks}, as expected.
\end{Example}
A lot of the constructions in this note will be done relative to a fixed stack, that is, will take place in the overcategory $\stacks _{/\euX}$. This $\infty$-category can be itself described as an $\infty$-category of sheaves in a standard way, as we now describe. 

\begin{Remark}
\label{Remark:description_of_overcategory_of_a_stack_as_sheaves}
If $\C$ is a small $\infty$-category and $\euX \in \presheaves(\C)$ is a presheaf, there is a canonical equivalence 
\begin{equation*}
\presheaves(\C)_{/ \euX} \simeq  \presheaves(\C_{/\euX}),
\end{equation*}
between the overcategory of the presheaves and presheaves on the overcategory \cite{higher_topos_theory}[5.1.6.12]. Under this equivalence, an object $\eF \in \presheaves(\C)_{/\euX} $ corresponds to an object which assigns to a morphism $f: U \rightarrow \euX(U)$, which we can identify with a point $f \in X(U)$, the fiber product $\{f\} \times_{\euX(U)} \widetilde{\eF}(U)$. Moreover, if $\C$ is an $\infty$-site, then there is an induced topology on the overcategory $\C_{/\euX}$ and the above equivalence restricts to one of the form $\sheaves(\C)_{/\euX} \simeq \sheaves(\C_{/\euX})$. 

In our situation, we will take $\C$ to be the site $\affines$, and so we deduce that for an arbitrary stack $\euX \in \stacks$ there is a canonical equivalence $\stacks_{/\euX} \simeq \sheaves(\affines_{/\euX})$. In this note we will use this equivalence implicitly, blurring the distinction between the two $\infty$-categories.  
\end{Remark}

\subsection{Quasi-coherent sheaves}

If $R$ is a ring, we can associate to it the derived $\infty$-category $\dcat(R)$, which is an $\infty$-categorical enhancement of the classical unbounded derived category. This $\infty$-category is stable and admits a canonical $t$-structure whose heart $\dcat(R)^{\heartsuit} := \dcat(R) _{\geq 0} \cap \dcat(R) _{\leq 0}$ is given by the the abelian category of $R$-modules. 

One advantage of working with stable $\infty$-categories, rather than triangulated categories, is that the former can be glued together in a controlled manner. This allows one to give a transparent definition of a quasi-coherent sheaf on an Artin stack, which we now review. 

\begin{Definition}
\label{Definition:qcoh_sheaves_on_a_stack}
Let $\euX \in \stacks$ be a stack. We define the stable $\infty$-category $\qcoh(\euX)$ of quasi-coherent sheaves on $\euX$ as the limit 
\begin{center}
$\qcoh(\euX) := \underset{\spec(A) \rightarrow \euX} {\varprojlim} \dcat(A)$
\end{center}
taken over the category of affine schemes equipped with a map into $\euX$, with the maps between module $\infty$-categories given by extension of scalars. 
\end{Definition}

\begin{Example}
If $\euX \simeq \spec(A)$ is affine, then the $\infty$-category of affines over $\euX$ has a terminal object given by the identity and we obtain 

\begin{center}
$\qcoh(\spec(A)) \simeq \dcat(A)$.
\end{center}
In particular, notice that according to this convention a quasi-coherent sheaf on $\spec(A)$ is an object of the derived $\infty$-category rather than a discrete $A$-module. 
\end{Example}
According to \cref{Definition:qcoh_sheaves_on_a_stack}, a quasi-coherent sheaf $\eF$ on $\euX$ consists of an assignment of an object $\eF(\spec(A)) \in \dcat(A)$ for each map $\eta: \spec(A) \rightarrow \euX$, equivalently, for each point $\eta \in X(A)$. This data is required to be compatible in the sense that we have distinguished equivalences $B \otimes _{A} \eF(\spec(A)) \simeq \eF(\spec(B))$ for each composite $\spec(B) \rightarrow \spec(A) \rightarrow \euX$, as well as higher coherence data. 

More formally, we define $\qcoh(\euX)$ as follows. One can construct an $\infty$-category $\dcat$ whose objects are pairs $(A, M)$, where $A \in \calg_{k}$ is a discrete $k$-algebra and $M \in \dcat(A)$, and such that the obvious functor $\dcat \rightarrow \calg_{k}$ is a coCartesian fibration. Then, $\qcoh(\euX)$ is given by the $\infty$-category of CoCartesian sections of the pullback fibration $\dcat \times _{\calg_{k}} (\calg_{k})_{/\euX} \rightarrow (\calg_{k})_{/\euX}$. 

\begin{Example}
The \emph{structure sheaf} $\mathcal{O}_{\euX}$ of a stack $\euX$ is the quasi-coherent sheaf given by 

\begin{center}
$\mathcal{O}_{\euX}(\spec(A) \rightarrow \euX) := A$,
\end{center}
where we consider the right hand side as an element of the heart of $\dcat(A)$.
\end{Example}

\begin{Remark}
Note that \cref{Definition:qcoh_sheaves_on_a_stack} makes sense already when $\euX$ is a prestack, but one can show that the $\infty$-category of quasi-coherent sheaves is the same on a prestack and its stackification. In other words, formation of $\qcoh$ satisfies descent with respect to the \'{e}tale topology, in fact, even with respect to the flat topology \cite{spectral_algebraic_geometry}[6.2.3.1].
\end{Remark}

\begin{Remark}
If $\euX$ is Artin, then one can replace the indexing $\infty$-category in \cref{Definition:qcoh_sheaves_on_a_stack} by the category of those affines $\spec(A) \rightarrow \euX$ which are smooth over $\euX$, see \cite{GR1}[1.4.2].
\end{Remark}

As a limit of stable, presentable $\infty$-categories, $\qcoh(\euX)$ is stable and presentable for any stack $\euX$. Moreover, it is functorial; for any morphism $f: \euX \rightarrow \euY$ of stacks we have an induced adjunction 

\begin{center}
$f^{*} \dashv f_{*}: \qcoh(\euY) \leftrightarrows \qcoh(\euX)$.
\end{center}
Using the informal description given above, $f^{*}$ is defined by $(f^{*} \eF)(\spec(A)) := \eF(\spec(A))$, and its right adjoint exists for abstract reasons. Notice in particular that if $f: \spec(A) \rightarrow \euX$ is a map from an affine scheme, then as an object of $\qcoh(\spec(A)) \simeq \dcat(A)$, the pullback $f^{*} \eF$ corresponds to $\eF(\spec(A))$. 

The $\infty$-category $\qcoh(\euX)$ admits a canonical $t$-structure in which $\eF$ is connective if and only if $\eF(\spec(A))$ is connective for any morphism $\spec(A) \rightarrow \euX$. In general, this $t$-structure is not well-behaved, but the situation is much better in the Artin case. 

\begin{Lemma}
\label{Lemma:t_structure_on_qcoh_of_an_artin_stack}
Let $\euX$ be Artin. Then, $\qcoh(\euX)$ admits a $t$-structure in which a quasi-coherent sheaf is (co)connective if and only if for any smooth atlas $p: \spec(U) \rightarrow \euX$, the quasi-coherent sheaf $p^{*} \eF$ is (co)connective. 
\end{Lemma}

\begin{proof}
This is \cite{lurie_thesis_dag}[5.2.4]. 
\end{proof}
Note that by definition, the pullback functor $f^{*}: \qcoh(\euY) \rightarrow \qcoh(\euX)$ preserves connective objects, which implies formally that its right adjoint $f_{*}$ preserves coconnective objects. If $f$ is a smooth morphism of Artin stacks, then by \cref{Lemma:t_structure_on_qcoh_of_an_artin_stack} above $f^{*}$ also preserves coconnectivity.

\begin{Definition} 
\label{Definition:quasi_coherent_sheaves_with_a_given_property}
Let $P$ be a property of objects of the derived $\infty$-category of a ring which is stable under arbitrary base-change. Then, if $\euX$ is a stack and $\eupsilon \in \qcoh(\euX)$, we say $\eupsilon$ \emph{has property $P$} if $f^*\eupsilon \in \dcat(A)$ has property $P$ for any $f: \spec(A) \rightarrow \euX$. 
\end{Definition}

The properties to which \cref{Definition:quasi_coherent_sheaves_with_a_given_property} applies which will be of interest to us are the properties of being perfect, perfect of given amplitude and perfect up to order $n$. These properties can be defined in a homotopy-invariant way, see \cite{higher_algebra}, \cite{spectral_algebraic_geometry}, but for the convenience of the reader we will rephrase them in terms of chain complexes. we remind the reader that we are using \emph{homological} grading conventions. 

\begin{Definition} 
Let $A$ be a discrete $k$-algebra and $M$ an object of $\dcat({A})$. We say that $M$ is
\begin{enumerate}
\item \emph{perfect} if it can be represented by a bounded chain complex of finitely generated projectives. 
\item \emph{perfect of amplitude $[a, b]$} if it can be represented by a bounded chain complex of finitely genereated projectives which can be chosen to vanish outside of degrees $d \in [a, b]$.
\item \emph{perfect to order $n$} if it can be represented by a chain complex which is bounded from below and consists of finitely generated projectives in degrees $d \leq n$.
\end{enumerate}
\end{Definition} 

\subsection{Cotangent complex}

Throughout the paper, we will need some basic properties of the cotangent complex. The latter is most naturally defined and constructed in the setting of derived algebraic geometry, and since several thorough references exist in the latter context, we will keep our exposition to the minimum. 

\begin{Remark}[Derived algebraic geometry in this paper]
While we introduce the cotangent complex in the language of derived algebraic geometry, as is most natural, the rest of the paper will not use any derived notions. A reader already familiar with the cotangent complex can safely skip this subsection. 
\end{Remark}

The construction of the cotangent complex for commutative rings first appeared in the work of Andr\'e and Quillen \cite{Andre}, \cite{Quillen}. Roughly speaking given a commutative $k$-algebra $A$ the cotangent complex $L_{A/k}$ controls the deformation theory $A$. More precisely there is a bijection between $\Ext_A^1 (L_{A/k}, A)$ and infinitesimal deformations of $A$ and the obstruction to pass to a higher order deformations lives in $\Ext_A^2(L_{A/k}, A)$. This was later globalized to the case of schemes by Grothendieck and Illusie \cite{Grothendeick_cotangent}, \cite{Illusie}.

In order to construct the cotangent complex of (higher) Artin stacks rigorously, it is useful to be able to find a universal property that characterizes it uniquely. To this end we take our inspiration from the K\"ahler differentials:

The K\"ahler differentials enjoy a universal property which stems from the fact that given a $k$-algebra $A$ and an $A$-module $M$ the $A$-module $\Omega_{A/k}$ corepresents $k$-linear derivations from $A$ to $M$. Already in this case it follows from the work of Andr\'e and Quillen that
\begin{equation*}
   \hom_{\calg_{A}} (A, A\oplus M) \simeq \hom_A(\Omega_{A/k}, M) \simeq \Ext^0 _A (L_{A/k}, M).
\end{equation*}
Moreover, they show that 
\begin{equation*}
    \hom_{\mathcal{H}(A)} (A, A \oplus M[i]) \simeq \Ext^i _A(L_A, M)
\end{equation*}
Where $\mathcal{H}(A)$ is the homotopy category of simplicial $A$-algebras and $  A \oplus M[i]$ is the trivial square zero extension of $A$ by the Eilenberg-Maclane space $K(M, i)$. In particular $A \oplus M[i]$ is generally not discrete. This suggests that to formulate a universal property for the cotangent complex simplicial commutative rings and $\infty$-categories must come into the picture even if we only care about discrete commutative $k$-algebras. 

A \emph{derived stack} is an \'{e}tale sheaf on the opposite of the $\infty$-category $\calg_{k}^{an}$ of \emph{animated $k$-algebras}, where the latter is the $\infty$-category underlying the model category of simplicial commutative $k$-algebras \cite{spectral_algebraic_geometry} \cite{TV2}. Any commutative $k$-algebra determines a discrete animated ring, and through left Kan extension one obtains a functor $\iota: \stacks \hookrightarrow d\stacks$ which can be shown to be fully faithful. Moreover, for any stack $\euX$ in our sense we have $\qcoh(\euX) \simeq \qcoh(\iota \euX)$.

\begin{Definition}
\cite{spectral_algebraic_geometry}[17.2.4.2] One says that a morphism $\euX \rightarrow \euY$ of derived stacks \emph{admits an algebraic cotangent complex} if there exists an almost connective quasi-coherent sheaf $L_{\euX / \euY} \in \qcoh(\euX)$ such that for any animated $k$-algebra $A$, any point $\eta \in \euX(A)$ and any $M \in \dcat(A)_{\geq 0}$, there is a natural equivalence

\begin{center}
$\map_{\dcat(A)_{\geq 0}}(\eta^{*} L_{\euX / \euY}, M) \simeq \fib_{\eta} ( \euX(A \oplus M ) \rightarrow \euX(A) \times_{\euY(A)} \euY(A \oplus M))$.
\end{center}
\end{Definition}

In other words, the algebraic cotangent complex $L_{\euX / \euY}$ corepresents derivations in animated $k$-algebras.

\begin{Definition}
\label{Definition:cotangent_complex_of_a_morphism_of_stacks}
If $\euX \rightarrow \euY$ is a morphism of stacks, then we say it \emph{admits a cotangent complex} if the associated morphism $\iota \euX \rightarrow \iota \euY$ of derived stacks admits an algebraic cotangent complex. In this case, the \emph{cotangent complex} $L_{\euX / \euY}$ is the image of $L_{\iota \euX / \iota \euY}$ under the equivalence $\qcoh(\euX) \simeq \qcoh(\iota \euX)$.
\end{Definition} 
It follows from our definition that if $f: \euX \rightarrow \euY$ and $g: \euY \rightarrow \euZ$ are a composable pair of morphisms of stacks which admit cotangent complexes, then we have a canonical cofibre sequence 

\begin{center}
$f^{*} L_{\euY / \euZ} \rightarrow L_{\euX / \euZ} \rightarrow L_{\euX / \euY}$
\end{center}
of quasi-coherent sheaves on $\euX$ \cite{spectral_algebraic_geometry}[17.2.5.2]. However, some care must be taken with base-change properties.

\begin{Warning}
The inclusion $\calg_{k} \hookrightarrow \calg_{k}^{an}$ of $k$-algebras into animated $k$-algebras does not preserves pushouts, which are given by the tensor product in the source and the derived tensor product in the target. It follows that the embedding $\iota: \stacks \hookrightarrow d\stacks$ of stacks into derived stacks does not preserves pullbacks, and the cotangent complex of \cref{Definition:cotangent_complex_of_a_morphism_of_stacks} does not satisfy arbitrary base-change in the same way its derived analogue does. 
\end{Warning}

\begin{Remark}
\label{Remark:cotangent_complex_of_a_stack_satisfies_smooth_basechange}
One can show that the embedding $i: \stacks \hookrightarrow d\stacks$ commutes with pullbacks along all flat morphisms. It follows that the cotangent complex of \cref{Definition:cotangent_complex_of_a_morphism_of_stacks} satisfies flat base-change.
\end{Remark}
We will now give existence and finiteness statements for the cotangent complex.

\begin{Proposition}
\label{Proposition:relative_n_stacks_have_minus_n_connective_cotangent_complex}
Let $f: \euX \rightarrow \euY$ be a relative $n$-Artin stack. Then, $f$ admits a cotangent complex $L_{\euX / \euY}$ which is $(-n)$-connective and perfect to order $-1$. If $f$ is smooth, then $L_{\euX / \euY}$ is perfect of non-positive amplitude.  
\end{Proposition}

\begin{proof}
It is not difficult to see from our inductive definition that if $f$ is relative $n$-Artin, then the associated morphism $\iota (\euX \rightarrow \euY)$ of derived stacks is $n$-representable in the sense of Lurie. The two statements are then given by \cite{spectral_algebraic_geometry}[I.1.2.5.3, I. 1.3.3.7].
\end{proof}

\begin{Proposition}
\label{Proposition:cotangent_complex_perfect_to_order_0_for_locally_of_finite_type}
Suppose that $f: \euX \rightarrow \euY$ is morphism of Artin stacks which is locally of finite type. Then, $L_{\euX / \euY}$ is perfect to order $0$. 
\end{Proposition}

\begin{proof}
Since the cotangent complex satisfies smooth base-change, we may assume that $\euY = Y$ is affine. Choose a smooth surjection $p: X \rightarrow \euX$ from a scheme, it is then enough to show that $p^{*} L_{\euX / Y}$ is perfect to order $0$. We have a cofibre sequence 

\begin{center}
$p^{*} L_{\euX / Y} \rightarrow L_{X / Y} \rightarrow L_{X / \euX}$,
\end{center}
and since the last term is perfect of non-positive amplitude, we see it is enough to show that $L_{X / Y}$ is perfect to order $0$. 

Since our notion of perfect of order $0$ is local, we may further assume that $X \rightarrow Y$ is a morphism of affines schemes of finite type. The statement is then clear, since $L_{X / Y}$ is connective and $h_{0}(L_{X / Y}) \simeq \Omega _{X / Y}^{0}$ is finitely generated. 
\end{proof}

We close with a basic example of a calculation of the cotangent complex.

\begin{Example} 
We will identify the cotangent complex of the stack $\deloop^{2} \mathbb{G}_{m}$ of \cref{Example:example_of_higher_artin_stack}. Let $i: \spec(k) \rightarrow \deloop^{2} \mathbb{G}_{m}$ be the base-point and consider the diagram

\begin{center}
	\begin{tikzpicture}
		\node (TLL) at (-2, 2) {$\spec(k)$};
		\node (TL) at (0, 1.3) {$\deloop \mathbb{G}_m$};
		\node (TR) at (2, 1.3) {$\spec(k)$};
		\node (BL) at (0, 0) {$\spec(k)$};
		\node (BR) at (2, 0) {$\deloop^2 \mathbb{G}_m.$};
		
		\draw [->] [bend left] (TLL) to (TR);
		\draw [->] [bend right] (TLL) to (BL);
		\draw [dashed, ->] (TLL) -- (TL) node [midway, above] {$\delta$};
		\draw [->] (TL) -- (TR) ;
		\draw [->] (TL) -- (BL);
		\draw [->] (TR) -- (BR)node [midway, right] {$i$};
		\draw [->] (BL) -- (BR) node [midway, below] {$i$};
		\end{tikzpicture}
\end{center}
where the square is a pullback. It follows that $i$ is smooth, and so \cref{Remark:cotangent_complex_of_a_stack_satisfies_smooth_basechange} implies that $\delta^{*} L_{\deloop \mathbb{G}_{m} / \spec(k)} \simeq L_{\spec(k) / \deloop^{2}\mathbb{G}_{m}}$. We then deduce from the cofibre sequence of cotangent complexes that
\begin{center}
$ i^*L_{\deloop^2 \mathbb{G}_m} \simeq \delta^*L_{\deloop \mathbb{G}_m}[-1] \simeq k[-2] $. 
\end{center}
In fact, one can show more generally that for any smooth abelian group scheme $G$ we have $i^*L_{\deloop^n G} \simeq \mathfrak{g}^\vee[-n]$, where $\mathfrak{g}$ is the Lie algebra.
\end{Example}

\section{The abelian cone associated to a quasi-coherent sheaf} 

In this section we study the \emph{abelian cone functor}, a contravariant analogue of the $h^1/h^0$ functor of Behrend and Fantechi. The main result of this section is \cref{Theorem:cone_over_tor_bounded_qcoh_sheaf_is_artin_and_smooth_for_perfects}, which establishes that the abelian cone is Artin under certain conditions.

 \begin{Definition}
 \label{Definition:abelian_cone_associated_to_a_qcoh_sheaf}
 Let $\euX$ be a stack and $\eupsilon \in \qcoh(\euX)$ a quasi-coherent sheaf. Then, the \emph{abelian cone} associated to $\eupsilon$ is the prestack over $\euX$ defined by the formula 
 
\begin{center}
$\mathbb{V}_{\euX} (\spec(A) \overset{f}{\rightarrow} \euX) := \map_{\dcat(A)} ( f^* \eupsilon , A),$
\end{center}
where the latter is the mapping space in $\qcoh(\spec(A)) \simeq \dcat(A)$. This construction yields the abelian cone functor $\mathbb{V}_{\euX}: \qcoh (\euX)^{\op} \rightarrow \prestacks_{/\euX}$.
\end{Definition}

We hope the next example eases the reader with the knowledge this is a straightforward extension of the construction of relative spectrum over a scheme. 

\begin{Example}
\label{Example:cone_of_a_classical_qcoh_sheaf_on_a_scheme} 
Let $S$ be a scheme and $\eupsilon \in \qcoh(S)^\heartsuit$ be a quasi-coherent sheaf in the classical sense. We claim that $\mathbb{V}_{S}(\eupsilon)$ is the relative spectrum of the symmetric algebra on $\eupsilon$.

We have $\mathbb{V}_{S}(\eupsilon)(\spec(A) \rightarrow S) = \map_{\dcat({A})}(f^{*} \eupsilon, A)$, where $f^{*}: \qcoh(S) \rightarrow \dcat(A)$ is the pullback functor between stable $\infty$-categories, which classically corresponds to the derived pullback in the sense that $h_{i} (f^{*} \eupsilon) \simeq R_{i} f^{*} \eupsilon$. Since $\eupsilon$ is connective, so is $f^{*} \eupsilon$, and since $A$ is a discrete commutative ring, it is also coconnective when considered as a module over itself. Since $(f^{*} \eupsilon)_{\leq 0} \simeq R_{0}f^{*} \eupsilon$, we deduce using the $t$-structure axioms that 

\begin{center}
$\map_{\dcat(A)} (f^*\eupsilon, A) \simeq \hom_{\Mod_{A}}(R_{0}f^{*} \eupsilon, A) \simeq \map_{\text{Sch}_{/S}}(\spec(A), \spec_{S}(\sym(R_{0}f^{*}\eupsilon)))$,
\end{center}
which is what we wanted to show. 
\end{Example}

\begin{Proposition}
Let $\euX$ be a stack. Then, $\mathbb{V}_{\euX}(\eupsilon)$ is a stack for any $\eupsilon \in \qcoh(\euX)$. 
\end{Proposition}

\begin{proof} We have to show that for any commutative diagram 
\begin{center}
\begin{tikzpicture}
		\node (TL) at (0, 1) {$\spec (A)$};
		\node (TR) at (3, 1) {$\spec(B)$};
		\node (BC) at (1.5, 0) {$\euX$};

		\draw [->] (TL) -- (TR) node [midway, above] {$p$};
		\draw [->] (TL) -- (BC) node [midway, below] {$f$};
		\draw [->] (TR) -- (BC) node [midway, below] {$g$};
	\end{tikzpicture},
\end{center}
where $p$ is an \'{e}tale covering of rings, the diagram 

\begin{center}
$\mathbb{V}_{\euX}(\spec(A) \rightarrow \euX) \rightarrow \mathbb{V}_{\euX}(\spec(B) \rightarrow \euX) \rightrightarrows \mathbb{V}_{\euX}(\spec(B \otimes _{A} B) \rightarrow \euX) \triplerightarrow \ldots$
\end{center}
induced by the \v{C}ech nerve of $p$ is a limit diagram of spaces. Unwinding the definitions, we see that we have to prove that 

\begin{center}
$\map_{\dcat(B)}(g^* \eupsilon, B) \rightarrow \map_{\dcat(A)} (f^*\eupsilon, A) \rightrightarrows  \map_{\dcat({A \otimes_B A}) } ((f \otimes f)^*\eupsilon, A \otimes_B A) \triplerightarrow \cdots$ 
\end{center}
is a limit. Using the adjunction between pullback and pushforward, it is enough to verify that 

\begin{center}
$g_*B \rightarrow  f_*A \rightrightarrows (f \otimes f)_*(A \otimes_B  A) \triplerightarrow \cdots$
\end{center}
is a limit diagram in $ \qcoh(\euX)$. Since $ g_*$ is a right adjoint and $f_{*} \simeq g_{*} \circ p_{*}$, this follows from the classical fact that if $p$ is faithfully flat, in particular \'{e}tale, then 

\begin{center}
$B \rightarrow p_*A \rightrightarrows (p \otimes p)_*(A \otimes_B  A) \triplerightarrow \cdots$
\end{center}
is a limit of $B$-modules. 
\end{proof}

Notice that since the mapping space between any two $A$-modules admits a canonical lift to a connective spectrum, in fact a connective $\mathbb{Z}$-module, the cone $C_{\euX}(\eupsilon)$ is canonically an abelian stack over $\euX$, that is, an abelian group object in $\stacks_{/\euX}$.  

\begin{Lemma}
\label{Lemma:cone_functor_takes_colimits_to_limits}
The cone functor $\mathbb{V}_{\euX}: \qcoh(\euX)^{op} \rightarrow \ab(\stacks _{/ \euX})$ takes colimits to limits.
\end{Lemma}

\begin{proof}
It is enough to observe that for any $f: \spec(A) \rightarrow \euX$ we have 

\begin{center}
$\mathbb{V}_{\euX}(\varinjlim \eupsilon_{i})(f) \simeq \map_{A}(f^{*}\varinjlim \eupsilon_{i}, A) \simeq \map_{A}(\varinjlim f^{*}\eupsilon_{i}, A) \simeq \varprojlim \map_{A}(\eupsilon_{i}, A) \simeq \varprojlim \mathbb{V}_{\euX}(\eupsilon_{i})(f)$,
\end{center}
where we've used that $f^{*}: \qcoh(\euX) \rightarrow \dcat({A})$ is a left adjoint.
\end{proof} 

\begin{Lemma} 
\label{Lemma:base_change_for_cone_functor} 
The cone construction satisfies base change in the sense that for any morphism $ \varphi : \euX \longrightarrow \euY$ of stacks and $\eupsilon \in \qcoh(\euY)$ there's a canonical equivalence $\euX \times _{\euY} \mathbb{V}_{\euY}(\eupsilon) \simeq \mathbb{V}_{\euX}(\varphi^{*} \eupsilon)$.  
\end{Lemma}

\begin{proof}
For any $f: \spec(A) \rightarrow \euX$ we have

\begin{center}
$(\euX \times _{\euY} \mathbb{V}_{\euY}(\eupsilon))(f) \simeq \mathbb{V}_{\euY}(\varphi \circ f) \simeq \map_{A}((\varphi \circ f)^{*} \eupsilon, A) \simeq \map_{A}(f^{*} \varphi^{*} \eupsilon A) \simeq \mathbb{V}_{\euX}(\varphi^{*} \eupsilon)(f)$,
\end{center}
which is what we wanted to show. 
\end{proof}

\begin{Lemma}
\label{Lemma:cone_of_qcoh_sheaf_only_depends_on_the_coconnective_part}
Let $\euX$ be a stack and $\eupsilon \in \qcoh(\euX)$. Then, $\mathbb{V}_{\euX}(\eupsilon) \simeq \mathbb{V}_{\euX}(\eupsilon_{\leq 0})$. 
\end{Lemma}

\begin{proof} 
For $f: \spec(A) \rightarrow \euX$ we have 

\begin{center}
$\mathbb{V}_{\euX}(\eupsilon)(f) \simeq \map_{A}(f^{*} \eupsilon, A) \simeq \map_{\qcoh(\euX)}(\eupsilon, f_{*}A)$
\end{center}
and since $A$ is coconnective as a module over itself, the same is true for $f_{*}A$ and we write further 

\begin{center}
$\map_{\qcoh(\euX)}(\eupsilon, f_{*}A) \simeq \map_{\qcoh(\euX)}(\eupsilon_{\leq 0}, f_{*}A) \simeq \map_{\qcoh(\euX)}(f^{*} \eupsilon_{\leq 0}, A) \simeq \mathbb{V}_{\euX}(\eupsilon_{\leq 0})(f)$.
\end{center}
\end{proof}

\begin{Lemma}
\label{Lemma:abelian_cone_is_affine_if_and_only_qcoh_sheaf_connective}
Let $\euX$ be a stack and $\eupsilon \in \qcoh(\euX)$. If $\eupsilon$ is connective, then $\mathbb{V}_{\euX}(\eupsilon) \rightarrow \euX$ is affine. If $\eupsilon$ is bounded below, then the converse holds. 
\end{Lemma}

\begin{proof}
By \cref{Lemma:base_change_for_cone_functor}, we can assume that $\euX \simeq \spec(A)$ is affine. To see the forward direction, observe that by \cref{Lemma:cone_of_qcoh_sheaf_only_depends_on_the_coconnective_part}, $\mathbb{V}_{\spec(A)}(\eupsilon) \simeq \mathbb{V}_{\spec(A)}(\eupsilon_{\leq 0})$. Since $\eupsilon_{\leq 0} \in \qcoh(\spec(A))^{\heartsuit}$ by assumption, the statement follows from \cref{Example:cone_of_a_classical_qcoh_sheaf_on_a_scheme}.

Now assume that $\eupsilon$ is bounded below and that $\mathbb{V}_{\euX}(\eupsilon) \rightarrow \euX$ is affine, in particular it has discrete fibers. We first claim that for an arbitrary map $ f: \spec(B) \rightarrow \euX$ and a $B$-module $M$, the space $\map_{\dcat(B)} (f^*\eupsilon , M)$ is discrete. 

If we form the trivial square-zero extension $B \oplus M$ and consider the composite morphism $g: \spec(B \oplus M) \rightarrow \euX$, it follows that the space
\begin{center}
$\mathbb{V}_\euX (\eupsilon)(\spec(B \oplus M)) = \map_{\dcat({B \oplus M})} ((B \oplus M ) \otimes^{L}_B f^*\eupsilon , B \oplus M) $
\end{center}
is discrete. Using the extension of scalars adjunction we can rewrite the right hand side as 

\begin{center}
$\map_{\dcat(B)} (f^*\eupsilon , B \oplus M) \simeq \map_{\dcat(B)} (f^*\eupsilon , B) \times \map_{\dcat(B)} (f^*\eupsilon , M)$
\end{center}
and thus we deduce that $\map_{\dcat(B)} (f^*\eupsilon , M)$ is discrete, as claimed.

Now suppose for a contradiction that $f^*\eupsilon$ is not connective and let $r < 0$ be the smallest integer such that $h_{r} (f^*\eupsilon) \neq 0$. Then 
\begin{center}
$\pi_{-r} \map_{\dcat(B)} (f^*\eupsilon, M) \simeq \pi_{0} \map_{\dcat(B)}((f^{*} \eupsilon)_{\leq r}, M) \simeq \Ext^{0}_{B}(h_{r} (f^*\eupsilon) , M)$
\end{center}
and since we've already shown that the left hand side vanishes for any $M$, we deduce that the same must be true for the right hand side. It follows that $h_{r}(f^{*} \eupsilon) = 0$, giving the desired contradiction and ending the argument.
\end{proof}

We are nearing the main result of this section, which states that the cone functor produces higher Artin stacks when it is applied to quasi-coherent sheaves satisfying certain finiteness conditions. We begin with a simple lemma.

\begin{Lemma}
\label{Lemma:abelian_cone_takes_certain_maps_of_qcoh_sheaves_to_atlases}
Let $\euX$ be a stack and let $\eupsilon \rightarrow \eupsilon^{\prime} \rightarrow \euP$ be a cofibre sequence of quasi-coherent sheaves on $\euX$ such that $\euP$ is perfect and of non-positive amplitude. Then, $\mathbb{V}_{\euX}(\eupsilon^{\prime}) \rightarrow \mathbb{V}_{\euX}(\eupsilon)$ is surjective and $\mathbb{V}_{\euX}(\eupsilon^{\prime}) \times _{C_{\euX}(\eupsilon)} \mathbb{V}_{\euX}(\eupsilon^{\prime}) \simeq \mathbb{V}_{\euX}(\eupsilon^{\prime}) \times_{\euX} \mathbb{V}_{\euX}(\euP)$.
\end{Lemma}

\begin{proof}
By \cref{Lemma:base_change_for_cone_functor}, we can assume that $\euX \simeq \spec(A)$ is affine. If $f: \spec(B) \rightarrow \spec(A)$ is a morphism of affine schemes, then the cofibre sequence $\Sigma^{-1} \euP \rightarrow \eupsilon \rightarrow \eupsilon^{\prime}$ induces a fibre sequence 

\begin{center}
$\mathbb{V}(\eupsilon^{\prime})(f) \rightarrow \mathbb{V}(\eupsilon)(f) \rightarrow \mathbb{V}(\Sigma^{-1} \euP)(f)$
\end{center}
of spaces. By definition, we have $\pi_{0} \mathbb{V}(\Sigma^{-1} \euP)(f) \simeq \Ext^{1}_{B}(f^{*} \euP, B)$ and the latter group vanishes, since $B$ is connective and $f^{*} \euP$ is perfect and of non-positive amplitude. Through the long exact sequence of homotopy we deduce that 

\begin{center}
$\pi_{0}\mathbb{V}(\eupsilon^{\prime})(\spec(B) \rightarrow \spec(A)) \rightarrow \pi_{0} \mathbb{V}(\eupsilon)(\spec(B) \rightarrow \spec(A))$
\end{center}
is surjective, proving the first claim. 

The second claim follows from the fact that $\eupsilon^{\prime} \oplus _{\eupsilon} \eupsilon^{\prime} \simeq \eupsilon^{\prime} \oplus \euP$ and \cref{Lemma:cone_functor_takes_colimits_to_limits}.
\end{proof}

\begin{Remark}
\label{Remark:map_of_qcoh_sheaves_with_coconnective_perfect_cofibre_induces_cofibre_seq_of_abelian_stacks}
The conclusion of \cref{Lemma:abelian_cone_takes_certain_maps_of_qcoh_sheaves_to_atlases} can be alternatively rephrased as follows: if $\eupsilon \rightarrow \eupsilon^{\prime} \rightarrow \euP$ is a fibre sequence of quasi-coherent sheaves with $\euP$ perfect and of non-positive amplitude, then $\mathbb{V}_{\euX}(\euP) \rightarrow \mathbb{V}_{\euX}(\eupsilon^{\prime}) \rightarrow \mathbb{V}_{\euX}(\eupsilon)$ is a cofibre sequence of abelian stacks over $\euX$. 
\end{Remark}

\begin{Theorem} 
\label{Theorem:cone_over_tor_bounded_qcoh_sheaf_is_artin_and_smooth_for_perfects}
Let $\euX$ be  a stack and let $\eupsilon \in \qcoh(\euX)$ be perfect to order $-1$ and $(-n)$-connective, where $n \geq 0$. Then, the morphism $\mathbb{V}_{\euX}(\eupsilon) \rightarrow \euX$ is a relative $n$-Artin stack. If $\eupsilon$ is perfect and of non-positive amplitude, then $\mathbb{V}_{\euX}(\eupsilon)$ is smooth. 
\end{Theorem}

\begin{proof} 
By \cref{Lemma:base_change_for_cone_functor}, we may assume that $\euX \simeq \spec(A)$ is affine. To declutter the notation, for $\eupsilon \in \qcoh(\spec(A))$ let us denote the cone by $\mathbb{V}(\eupsilon) := \mathbb{V}_{\spec(A)}(\eupsilon)$. 

By \cref{Lemma:cone_of_qcoh_sheaf_only_depends_on_the_coconnective_part}, we can replace $\eupsilon$ by its coconnective cover. Since $\eupsilon$ is perfect to order $-1$ and $(-n)$-connective, it follows that $\eupsilon_{\leq 0} \in \qcoh(\spec(A)) \simeq \dcat(A)$ can be represented by a complex 

\begin{center}
$\ldots \rightarrow 0 \rightarrow E_{0} \rightarrow E_{-1} \rightarrow \ldots \rightarrow E_{-n} \rightarrow 0 \rightarrow \ldots$
\end{center}
of $A$-modules, where $E_{i}$ are finitely generated, projective for $i < 0$. If $\eupsilon$ is perfect and of non-positive amplitude, then we can additionally assume that $E_{0}$ is projective as well. 

Our proof will go by induction on $n$, the base case $n = 0$ following from \cref{Lemma:abelian_cone_is_affine_if_and_only_qcoh_sheaf_connective}. Thus, suppose that $n > 0$ and consider the cofibre sequence

\begin{center}
$\eupsilon \rightarrow E_0 \rightarrow \euP$
\end{center}
in $\qcoh(\spec(A))$, where $\euP \simeq [P_{-1} \rightarrow \ldots \rightarrow P_{-n}]$ with $P_{-1}$ concentrated in degree zero. Notice that $\euP$ is perfect and coconnective and so it follows from \cref{Lemma:abelian_cone_takes_certain_maps_of_qcoh_sheaves_to_atlases} that the map $\mathbb{V}(E_{0}) \rightarrow \mathbb{V}(\eupsilon)$ is surjective and 

\begin{center}
$\mathbb{V}(E_{0}) \times _{\mathbb{V}(\eupsilon)} \mathbb{V}(E_{0}) \simeq \mathbb{V}(E_{0}) \times_{\euX} \mathbb{V}(\euP)$.
\end{center}
Since $\mathbb{V}(\euP) \rightarrow \euX$ is smooth $(n-1)$-Artin by inductive assumption, the needed result follows.  
\end{proof}

\begin{Corollary}
Let $\euX$ be an Artin and let $\eupsilon$ be perfect to order $-1$. Then, $\mathbb{V}_{\euX}(\eupsilon)$ is Artin.
\end{Corollary}

\begin{proof}
Immediate from \cref{Theorem:cone_over_tor_bounded_qcoh_sheaf_is_artin_and_smooth_for_perfects} and \cref{Proposition:properties_of_artin_morphisms}. 
\end{proof}

As mentioned at the beginning of the section, the abelian cone should be thought of as an analogue of the $h^{1}/h^{0}$ functor defined by Behrend and Fantechi, itself inspired by Deligne's work on Picard stacks. The following remark explores this connection futher.

\begin{Remark}
\label{Remark:comparision_of_h^1/h^0_with_cone}
A careful reader will notice quickly that $h^1/h^0$ is covariant, while the abelian cone functor of \cref{Definition:abelian_cone_associated_to_a_qcoh_sheaf} is contravariant. This is because the two are not really analogous, but rather dual to each other.

 The \emph{global section} functor $\Gamma$ which is the higher categorical analogue of $h^{1}/h^{0}$ associates to any $\eupsilon \in \qcoh(\euX)$ the prestack $\Gamma(\eupsilon)$ defined by the formula
 
 \begin{center}
 $(f: \spec(A) \rightarrow \euX) \mapsto \map_{\dcat(A)} (A, f^{*} \eupsilon)$
 \end{center}
One can show that the functor $ \Gamma$ is cocontinuous and takes values in stacks. One can then check directly that if $\euX$ is Deligne-Mumford stack, then the restriction $\Gamma |_{\qcoh(\euX)_{\leq 1}}$ coincides with the functor $h^1/h^0$ after suitable identifications.

When restricted to the full subcategory of almost perfect objects, the cone functor can be defined in terms of the global sections functor. That is, if $\euX$ is a stack, then there exists a commutative diagram
\begin{center}
	\begin{tikzpicture}
		\node (BL) at (0,0) {${{\qcoh(\euX)}^{\text{aperf}}}^{\op}$};
		\node (BC) at (6, 0) {$\qcoh(\euX)$};
		\node (BR) at (6, -1.3) {$ \stacks_{/\euX}$};

		\draw [->] (BL) -- (BC) node [midway, above] {$\underline{\map}_{\qcoh(\euX)}(- , \EuScript{O}_\euX)$};
		\draw [->] (BC) -- (BR) node [midway, right] {$\Gamma$};
		\draw [->] (BL) -- (BR) node [midway, below] {$\mathbb{V}_\euX(-)$};
	\end{tikzpicture}
\end{center}
where if $\eF \in \qcoh(\euX)$ is almost perfect then the quasi-coherent sheaf $ \underline{\map}_{\qcoh(\euX)}(\eF , \EuScript{O}_\euX) $ is defined so that we have a canonical equivalence 

\begin{center}
$\Omega^{\infty}\underline{{\map}}_{\qcoh(\euX)} ( \eF, \EuScript{O}_\euX) ( \spec(A)\overset{f}{\rightarrow} \euX) \simeq \map_{\dcat(A)} ( f^*\eF, A)$,
\end{center}
as in \cite{spectral_algebraic_geometry}[6.5.3]. 

If $\euX$ is Deligne-Mumford, then after taking different grading conventions into account, one sees that if $\eupsilon \in \qcoh(\euX)_{[0 ,1]}$ has coherent homology, the stack $ \Gamma(\underline{\map}_{\qcoh(\euX)} (\eupsilon[-1], \EuScript{O}_\euX))$ coincides with the one denoted by Behrend and Fantechi as $ h^1/ h^0 (\eupsilon^{\vee}) $. Thus, we deduce from the above analysis that the latter coincides with our $\mathbb{V}_\euX(\eupsilon[-1])$.
\end{Remark}

\section{Adapted cosheaves}
\label{Section:adapted_cosheaves}

In this section we introduce the notion of an \emph{adapted cosheaf}, which is a cosheaf which preserves pullbacks along a distinguished class of maps. Our main result is that such cosheaves are stable under left Kan extension to a larger category. 

\begin{Definition}
We will say that a class $S$ of morphisms in an $\infty$-category $\Ccat$ is a \emph{marking} if 

\begin{enumerate}
\item $S$ contains all equivalences and is stable under composition and 
\item $\Ccat$ admits pullbacks along morphisms in $S$ and $S$ is stable under such pullbacks.
\end{enumerate}
\end{Definition}
The usefulness of the concept of a marking is encapsulated in the following straightforward result. Recall that we say that an $\infty$-category $\Ccat$ has universal coproducts if for any $D^{\prime} \rightarrow D$ and a finite collection of maps $C_{i} \rightarrow D$, the canonical map $\bigsqcup C_{i} \times _{D} D^{\prime} \rightarrow (\bigsqcup C_{i}) \times _{D} D^{\prime}$ is an equivalence. 

\begin{Proposition}
\label{Proposition:marking_determines_grothendieck_topology}
Let $\Ccat$ be an $\infty$-category with universal coproducts and $S$ be a marking on $C$. Then, $\Ccat$ admits a Grothendieck topology in which a family $\{ C_{i} \rightarrow C \}$ is covering precisely when $\bigsqcup C_{i} \rightarrow C$ belongs to $S$. 
\end{Proposition}

\begin{proof}
This is \cite{spectral_algebraic_geometry}[A.3.2.1].
\end{proof}
The category of schemes, as well as its higher categorical variants, has universal coproducts and by taking $S$ to be an appropriate class of morphisms (such as surjective \'{e}tale) we can produce the classical Grothendieck topologies in schemes (such as the \'{e}tale topology). 

\begin{Definition}
\label{Definition:downward_closed_and_generating_subcategory}
We will say a full subcategory $\Dcat \subseteq \Ccat$ is \emph{downward closed} if it has the property that if $C_{0} \rightarrow C$ is a covering such that $C_{i} \in \Dcat$ for all $i \geq 0$, where $C_{i} := C_{0} \times _{C} \times \ldots \times _{C} C_{0}$, then $C \in \Dcat$ as well. 

We say a full subcategory $\Dcat \subseteq \Ccat$ is \emph{generating} if it is closed under coproducts, pullbacks along coverings, and the smallest downward closed subcategory which contains it is all of $\Ccat$. 
\end{Definition}
Note that since a generating subcategory is closed under coproducts and pullbacks along coverings, it inherits a topology. The definition is chosen so that the following is true. 

\begin{Proposition}
\label{Proposition:restriction_of_sheaves_to_generating_subcategory_is_an_equivalence}
Let $\Dcat \subseteq \Ccat$ be a generating subcategory. Then, the restriction $Sh(\Ccat) \rightarrow Sh(\Dcat)$ between sheaf $\infty$-categories is an equivalence.
\end{Proposition}

\begin{proof}
The inclusion $: \Dcat \hookrightarrow \Ccat$ induces an adjunction $i_{*} \dashv i^{*}: Sh(\Ccat) \leftrightarrows Sh(\Dcat)$, where $i^{*}$ is the restriction and $i_{*}$ is the left Kan extension. Clearly, we have $i^{*} \circ i_{*} \simeq id$.

To see that the other composite is also the identity, observe that the sheaf condition implies that for any $F \in Sh(\Ccat)$ the subcategory of those $C \in \Ccat$ such that the counit $i_{*} i^{*} F(C) \rightarrow F(C)$ is an equivalence, is downward closed. Since it contains all of $\Dcat$ which is generating by assumption, it must be all of $\Ccat$. 
\end{proof}

The functors considered in this paper will not only be cosheaves, but will additionally have compatibility with a wider class of morphisms than just coverings. This further class will also form a marking; in practice this could be either the class of all smooth or even all flat morphisms, not necessarily surjective. 

Throughout the rest of the section, which forms the technical heart of this paper, we will work in the following setup.

\begin{Notation}
By $\Ccat$ we denote an $\infty$-category with universal coproducts, and we assume that we are given a pair of markings $S \subseteq T$. The word \emph{covering} will mean a morphism in $S$, and we will treat $\Ccat$ as a Grothendieck site with respect to the topology determined by $S$.
\end{Notation}

\begin{Notation}
The symbol $\euX$ will denote an $\infty$-topos. Following standard terminology, we will say that $F: \Ccat \rightarrow \euX$ is a \emph{cosheaf} if the induced functor between opposite $\infty$-categories
\[
F^{op}: \Ccat^{op} \rightarrow \euX^{op}
\]
is a sheaf with respect to the Grothendieck topology determined by $S$. Equivalently, $F$ is a cosheaf if for any $x \in \euX$, the composite
\[
\map(F(-), x): \Ccat^{op} \rightarrow \spaces
\]
is a sheaf of spaces.
\end{Notation}

The purpose of the additional marking $T$ is to single out the following class of cosheaves. 

\begin{Definition}
\label{Definition:adapted_cosheaf}
Let $T$ be a marking such that $S \subseteq T$. We will say that a cosheaf $F: \Ccat \rightarrow \euX$ is $T$-\emph{adapted} if $F$ commutes with pullbacks along morphisms in $T$. 
\end{Definition}

It turns out that the pullback preservation is so strong that it almost implies the cosheaf condition, as the following shows. 

\begin{Proposition}
\label{Proposition:adapted_cosheaves_are_functors_that_preserve_pullbacks_and_surjections}
A functor $F: \Ccat \rightarrow \euX$ which commutes with pullbacks along morphisms in $T$ is a cosheaf if and only if 

\begin{enumerate}
\item $F$ preserves coproducts 
\item $F$ takes morphisms in $S$ to effective epimorphisms. 
\end{enumerate}
\end{Proposition}

\begin{proof}
As a consequence of the description of sheaves of \cite{spectral_algebraic_geometry}[A.3.3.1], $F$ is a cosheaf if and only if it takes \v{C}ech nerves 

\begin{center}
$\ldots \rightrightarrows C \times _{D} C \rightarrow C \rightarrow D$
\end{center}
of morphisms $C \rightarrow D$ into colimit diagrams in $\euX$. However, since by assumption $S \subseteq T$ and $F$ preserves pullbacks along elements of the latter, the above diagram is taken to the \v{C}ech diagram of $F(C) \rightarrow F(D)$. Thus, it is a colimit diagram precisely when $F(C) \rightarrow F(D)$ is an effective epimorphism. 
\end{proof}
Our main goal for this section is to prove that the property of being adapted with respect to a given marking $T$ can be verified on a subcategory. As a consequence, we conclude that the unique extension of an adapted cosheaf is adapted. We begin with a simple lemma. 

\begin{Lemma}
\label{Lemma:subcategory_of_morphisms_for_which_f_commutes_with_pullbacks_along_a_given_morphism_is_downward_closed}
Let $C \in \Ccat$ and fix a morphism $D \rightarrow C$. Then, the subcategory of those morphisms $E \rightarrow C$ such that 

\begin{center}
$F(D \times _{C} E) \simeq F(D) \times _{F(C)} F(E)$
\end{center}
is a downward closed subcategory of $\Ccat_{/C}$. 
\end{Lemma}

\begin{proof}
Suppose that we have a covering $E_{0} \rightarrow E$ such that $F(E_{i} \times _{C} D) \simeq F(E_{i}) \times _{F(C)} F(E)$ for all $i \geq 0$, where $E_{i} := E_{0} \times _{E} \ldots \times_{E} E_{0}$. Then, 

\begin{center}
$F(E \times _{C} D) \simeq \varinjlim F(E_{i} \times _{C} D) \simeq \varinjlim F(E_{i}) \times _{F(C)} F(D) \simeq F(E) \times _{F(C)} F(D)$,
\end{center}
where we have twice used that $F$ is a cosheaf and once that $\euX$ is an $\infty$-topos, so that pullbacks therein commute with colimits. 

\end{proof}

\begin{Theorem}
\label{Theorem:extension_of_an_adapted_cosheaf_is_adapted}
Let $\Dcat \subseteq \Ccat$ be a generating subcategory closed under coproducts and pullbacks along coverings. Then, a cosheaf $F: \Ccat \rightarrow \euX$ is $T$-adapted if and only if its restriction $F _{| \Dcat}$ is. 
\end{Theorem}

\begin{proof}
One direction is trivial, so instead suppose that $F$ is a cosheaf such that $F_{| \Dcat}$ is adapted, we have to show that $F$ is adapted. 

Let us say that a morphism $D \rightarrow C$ is \emph{good} if for any other morphism $E \rightarrow C$ we have $F(E \times _{C} D) \simeq F(E) \times _{F(C)} F(E)$; our goal is to show that any morphism in $T$ is good. If $D \rightarrow C$ is a morphism in $\Dcat$ which belongs to $T$, then since $F_{| \Dcat}$ is assumed to be adapted, we deduce that the condition holds whenever $E \in \Dcat$. Then, it follows from \cref{Lemma:subcategory_of_morphisms_for_which_f_commutes_with_pullbacks_along_a_given_morphism_is_downward_closed} that all such arrows are good.

Let us further say that an object $C$ is \emph{excellent} if any morphism $D \rightarrow C$ which belongs to $T$ is good, we claim that any $C \in \Dcat$ is excellent. By another application of \cref{Lemma:subcategory_of_morphisms_for_which_f_commutes_with_pullbacks_along_a_given_morphism_is_downward_closed} we see that the collection of good morphisms is downward closed in the $\infty$-category $\Ccat^{T}_{/C}$ of arrows $D \rightarrow C$ which belong to $T$, so that it is enough to verify it when we also have $D \in \Dcat$, which we already did.

Lastly, we claim that the collection of excellent objects of $\Ccat$ is also downward closed; since we already verified that it contains all objects of $\Dcat$, this will end the argument. Suppose that we have a covering $C_{0} \rightarrow C$ such that all of $C_{i} := C_{0} \times _{C} \ldots \times _{C} C_{0}$ are excellent, we want to show that the same is true for $C$. We first show that $F$ takes the \v{C}ech nerve of $C_{0} \rightarrow C$ to the \v{C}ech nerve of $F(C_{0}) \rightarrow F(C)$, in particular, that $F(C_{1}) \simeq F(C_{0}) \times _{F(C)} F(C_{0})$. Since $F$ is a cosheaf, the image 

\begin{center}
$\ldots F(C_{1}) \rightrightarrows F(C_{0}) \rightarrow F(C)$
\end{center}
is a colimit diagram. Because $\euX$ is an $\infty$-topos, by \cite{higher_topos_theory}[6.1.3.19, (iv)] to verify that the above diagram is a \v{C}ech nerve it is enough to check that the underlying simplicial object is a groupoid; in other words, that for any partition $[m] = S \cup T$ with $S \cap T = \{ s \}$, the induced diagram

\begin{center}
	\begin{tikzpicture}
		\node (TL) at (0, 1.3) {$ F(C_{m}) $};
		\node (TR) at (2, 1.3) {$ F(C_{| S |}) $};
		\node (BL) at (0, 0) {$ F(C_{| T |}) $};
		\node (BR) at (2, 0) {$ F(C_{0}) $};
		
		\draw [->] (TL) to (TR);
		\draw [->] (TL) -- (BL);
		\draw [->] (TR) -- (BR);
		\draw [->] (BL) to (BR);
	\end{tikzpicture}
\end{center}
is a pullback. This is clear, since $C_{0}$ is assumed to be excellent.

To check that $C$ itself is excellent, we have to verify that an arbitrary map $D \rightarrow C$ is good; by what was said above, it is enough to check that this is the case for $D_{i} \rightarrow C$, where $D_{i} := C_{i} \times _{C} D$. Thus, we can assume that the given map $D \rightarrow C$ factors through $C_{0}$. Then, by again invoking \cref{Lemma:subcategory_of_morphisms_for_which_f_commutes_with_pullbacks_along_a_given_morphism_is_downward_closed} we see that we only have to verify that $F(D \times _{C} E) \simeq F(D) \times _{C} F(E)$ where $E \rightarrow C$ also factors through $C_{0}$. 

To summarize, to prove that $C \in \Ccat$ is excellent, it is enough to show that $F$ preserves pullbacks of spans which can be factorized as 

\begin{center}
$D \rightarrow C_{0} \rightarrow C \leftarrow C_{0} \leftarrow E$,
\end{center}
where each map belongs to $T$. Consider the diagram 

\begin{center}
	\begin{tikzpicture}
		\node (TL) at (-2.5, 1.3) {$ F(D \times _{C} E) $};
		\node (TM) at (0, 1.3) {$ F(C_{1} \times _{C_{0}} E) $};
		\node (TR) at (2.5, 1.3) {$ F(E) $};
		\node (ML) at (-2.5, 0) {$ F(D \times _{C_{0}} C_{1}) $}; 
		\node (MM) at (0, 0) {$ F(C_{1}) $};
		\node (MR) at (2.5, 0) {$ F(C_{0}) $};
		\node (BL) at (-2.5, -1.3) {$ F(D) $}; 
		\node (BM) at (0, -1.3) {$ F(C_{0}) $};
		\node (BR) at (2.5, -1.3) {$ F(C) $};
		
		\draw [->] (TL) to (TM);
		\draw [->] (TM) to (TR);
		\draw [->] (ML) to (MM);
		\draw [->] (MM) to (MR);
		\draw [->] (BL) to (BM);
		\draw [->] (BM) to (BR);
		
		\draw [->] (TL) to (ML);
		\draw [->] (ML) to (BL);
		\draw [->] (TM) to (MM);
		\draw [->] (MM) to (BM);
		\draw [->] (TR) to (MR);
		\draw [->] (MR) to (BR);
	\end{tikzpicture},
\end{center}
where each of the squares except possibly the lower right one is a pullback because $C_{i}$ are assumed to be excellent. Because we also verified the same about the last square, the pullback pasting lemma ends the argument. 
\end{proof}
We will later find ourselves in a situation where we have a natural transformation between adapted cosheaves which is particularly nice when restricted to the subcategory. It will then be useful to know that this "niceness" necessarily holds in general, as we will now verify. 

\begin{Definition}
\label{Definition:tcartesian_transformation_between_cosheaves}
Let $F, G: \Ccat \rightarrow \euX$ be functors. Then, we say that a natural transformation $F \rightarrow G$ is \emph{$T$-cartesian} if for any arrow $D \rightarrow C$ in $T$ the induced diagram 

 \begin{center}
	\begin{tikzpicture}
		\node (TL) at (0, 1.3) {$ F(D) $};
		\node (TR) at (1.6, 1.3) {$ G(D) $};
		\node (BL) at (0 , 0) {$ F(C) $};
		\node (BR) at (1.6, 0){$ G(C) $};
		
		\draw[->] (TL) -- (TR);
		\draw[->] (BL) -- (BR);
		\draw[->] (TL) -- (BL);
		\draw[->] (TR) -- (BR);
	\end{tikzpicture}
\end{center}
is cartesian.
\end{Definition}

\begin{Lemma}
\label{Lemma:property_of_being_an_adapted_cosheaf_descends_along_cartesian_morphism}
Suppose that $F, G: \Ccat \rightarrow \euX$ are coproduct-preserving functors and let $F \rightarrow G$ be $T$-cartesian. Then, if $G$ is a $T$-adapted cosheaf, the same is true for $F$. 
\end{Lemma}

\begin{proof}
Suppose that we have a cospan $C_{0} \rightarrow C \leftarrow D$ where the left map belongs to $T$, we have to check that $F(D_{0}) \rightarrow F(\widetilde{C}) \times_{F(C)} F(D)$ is an equivalence, where $D_{0} := \widetilde{C} \times_{C} D$. Applying the cartesian property to both the source and target of this morphism, we see that this is equivalent to asking whether 

\begin{center}
$G(D_{0}) \times_{G(D)} F(D) \rightarrow G(C_{0}) \times _{G(C)} F(C) \times_{F(C)} F(D) \simeq G(C_{0}) \times _{G(C)} F(D) $
\end{center}
is an equivalence. Since $G$ is assumed to be adapted, we have $G(D_{0}) \simeq G(C_{0}) \times _{G(C)} G(D)$ and it follows that the source of the above map we can rewrite as 

\begin{center}
$G(D_{0}) \times_{G(D)} F(D) \simeq G(C_{0}) \times _{G(C)} G(D) \times_{G(D)} F(D)  \simeq  G(C_{0}) \times _{G(C)} F(D)$
\end{center}
which is what we wanted to show. 

By \cref{Proposition:adapted_cosheaves_are_functors_that_preserve_pullbacks_and_surjections}, to finish showing that $F$ is an adapted cosheaf, we just have to check that it takes coverings to effective epimorphisms. However, if $C_{0} \rightarrow C$ is a covering, then by the cartesian property we have $F(C_{0}) \simeq G(C_{0}) \times _{G(C)} F(C)$ and we deduce that $F(C_{0}) \rightarrow F(C)$ is a base-change of $G(C_{0}) \rightarrow G(C)$, which is an effective epimorphism since $G$ is a cosheaf. 
\end{proof}

\begin{Proposition} 
\label{Proposition:cartesian_natural_transformations_of_adapted_cosheaves_are_stable_under_extension}
Let $F, G: \Ccat \rightarrow \euX$ be $T$-adapted cosheaves, $F \rightarrow G$ be a natural transformation and let $\Dcat \subseteq \Ccat$ be a generating subcategory. Then, if the restriction $F _{| \Dcat} \rightarrow G_{| \Dcat}$ is $T$-cartesian, then so is $F \rightarrow G$. 
\end{Proposition}

\begin{proof}
We first claim that for any $C \in \Ccat$, the subcategory of those morphisms $D \rightarrow C$ such that $F(D) \simeq G(D) \times_{G(C)} F(C)$ is a downward closed subcategory of $\Ccat_{/C}^{T}$. To see this, let $D_{0} \rightarrow C$ be a covering such that all $D_{i} \rightarrow C$ have this property, where $D_{i} := D_{0} \times _{D} \ldots \times_{D} D_{0}$. Then,  

\begin{center}
$F(D) \simeq \varinjlim F(D_{i}) \simeq \varinjlim G(D_{i}) \times _{G(C)} F(C) \simeq G(D) \times _{G(C)} F(C)$,
\end{center}
where we have used that $F, G$ are cosheaves and that colimits in $\euX$ commute will pullbacks. We deduce that any morphism in $D \rightarrow C$ in $T$ with $C \in \Dcat$ has the required property. 

We next claim that the subcategory of those $C$ such that for any morphism $D \rightarrow C$ in $T$ we have $F(D) \simeq G(D) \times_{G(C)} F(C)$ is a generating subcategory of $\Ccat$, together with what we've shown above this will finish the argument. Choose a covering $C_{0} \rightarrow C$ such that $C_{0}$ has the needed property and let $D_{0} := C_{0} \times_{C} D$. Then, since $F$ and $G$ are adapted, we have 

\begin{center}
$F(D) \times _{F(C)} F(C_{0}) \simeq F(D_{0}) \simeq G(D_{0}) \times_{G(C_{0})} F(C_{0}) \simeq G(D) \times _{G(C)} G(C_{0}) \times_{G(C_{0})} F(C_{0}) $
\end{center}
and further

\begin{center}
$G(D) \times _{G(C)} G(C_{0}) \times_{G(C_{0})} F(C_{0}) \simeq G(D) \times _{G(C)} F(C_{0}) \simeq G(D) \times _{G(C)} F(C) \times _{F(C)} F(C_{0})$.
\end{center}
We deduce that we have $F(D) \simeq G(D) \times _{G(C)} F(C)$ after base-changing along $F(C_{0}) \rightarrow F(C)$ and since the latter is an effective epimorphism since $F$ is a cosheaf, we deduce that this holds even before the base-change, ending the proof.
\end{proof}

 \section{Axiomatization of the normal sheaf}
In this section, we study the notion of a normal sheaf of a morphism of Artin stacks, defined in terms of the cotangent complex. Our main result is that, as a functor on relative Artin stacks, the normal sheaf is determined by its values on closed embeddings of schemes and a short list natural axioms.

Recall that if $U \hookrightarrow V$ is a closed embedding of schemes with ideal sheaf $I$, then the cotangent complex $L_{U / V}$ is $1$-connective and $h_{1} (L_{U/ V}) \simeq I / I^{2}$. The abelian cone  associated to the latter quasi-coherent sheaf defines a scheme over $U$ known as the \emph{normal sheaf} of the embedding, denoted by $N_{U}V := \mathbb{V}_{U}(I / I^{2})$. In the particular case when the embedding is regular, $I / I^{2}$ is locally free and the normal sheaf is just the classical normal bundle. 

\begin{Definition}[\cite{behrend_fantechi_intrinsic_normal_cone}]
\label{Definition:normal_sheaf_of_a_morphism_of_artin_stacks}
Let $\euX \rightarrow \euY$ be a morphism of Artin stacks. Then, its \emph{normal sheaf} is defined as

\begin{center}
$\Ncat_{\euX} \euY := \mathbb{V}_{\euX}(L_{\euX / \euY}[-1])$,
\end{center}
the abelian cone associated to the shift of the cotangent complex.
\end{Definition}

\begin{Warning}
\label{warning:normal_sheaf_terminology}
The terminology \emph{normal sheaf} is very much established, going back to the original paper of Behrend and Fantechi on the intrinsic normal cone \cite{behrend_fantechi_intrinsic_normal_cone}. It is, however, potentially confusing - we do warn the reader that the normal sheaf is not a sheaf, but rather a stack over $\euX$. 
\end{Warning}

Let us give a couple of examples.
\begin{Example}
\label{Example:normal_sheaf_functor_on_closed_embedding_of_affines_is_same_as_classical}
If $U \hookrightarrow V$ is a closed embedding of schemes, then $\Ncat_{U}V$ coincides with the normal sheaf in the classical sense. To see this, note that we verified in \cref{Lemma:cone_of_qcoh_sheaf_only_depends_on_the_coconnective_part} that the abelian cone only depends on the coconnective part of a quasi-coherent sheaf, so that we have $\Ncat_{U} V \simeq \mathbb{V}_{U}((L_{U/ V}[-1])_{\geq 0}) \simeq \mathbb{V}_{U}(I / I^{2})$, where $I$ is the ideal sheaf. 
\end{Example}

\begin{Example}
If $\euX$ is Deligne-Mumford, then the normal sheaf of the unique map $\euX \rightarrow \spec(k)$ coincides with the \emph{intrinsic normal sheaf} of Behrend and Fantechi \cite{behrend_fantechi_intrinsic_normal_cone}[3.6]. As a particularly easy example of the latter, let us suppose that $\euX \simeq X$ is a smooth scheme. Then, $L_{X} \simeq \Omega_{X}$ and so $\Ncat_{X} \simeq \mathbb{V}_{X}(\Omega_{X}[-1]) \simeq \deloop T_{X}$; that is, there's an equivalence between the intrinsic normal sheaf of $X$ and the classifying stack of its tangent bundle.
\end{Example}
Note that a priori the normal sheaf of a morphism $\euX \rightarrow \euY$ of Artin stacks is a priori only a stack, and in fact it can fail to be algebraic unless we impose some finiteness conditions.

\begin{Proposition}
\label{Proposition:normal_sheaf_of_morphisms_of_artin_stacks_is_again_artin}
Let $\euX \rightarrow \euY$ be a morphism of Artin stacks which is locally of finite type. Then, $\Ncat_{\euX} \euY$ is Artin and moreover, 

\begin{enumerate}
\item if $\euX \rightarrow \euY$ is relative $n$-Artin, then $\Ncat_{\euX} \euY \rightarrow \euX$ is relative $(n+1)$-Artin,
\item if $\euX \rightarrow \euY$ is smooth, then so is $\Ncat_{\euX} \euY \rightarrow \euX$.
\end{enumerate}
\end{Proposition}

\begin{proof} 
This follows from \cref{Theorem:cone_over_tor_bounded_qcoh_sheaf_is_artin_and_smooth_for_perfects}, \cref{Proposition:cotangent_complex_perfect_to_order_0_for_locally_of_finite_type} and \cref{Proposition:relative_n_stacks_have_minus_n_connective_cotangent_complex}.
\end{proof}

\begin{Remark}
One can consider \cref{Proposition:normal_sheaf_of_morphisms_of_artin_stacks_is_again_artin} as giving a quantitative reason why Behrend and Fantechi only define the normal sheaf for morphisms of Deligne-Mumford type - if $\euX \rightarrow \euY$ is $1$-Artin, then the correct definition makes $\Ncat_{\euX} \euY$ into a $2$-Artin stack, forcing the introduction of higher algebraic stacks. 
\end{Remark}

\begin{Remark}
If $\euX \rightarrow \euY$ is a morphism of Artin stacks which is not locally of finite type, then $\Ncat_{\euX} \euY$ can fail to be Artin. Nevertheless, it is always "algebraic" in the sense that it can be locally obtained by starting from a scheme and taking quotients by actions of flat group schemes; what can fail is that without finiteness these group schemes will in general not be smooth. 
\end{Remark}
Our goal will be to prove that the normal cone functor is uniquely determined by a simple set of axioms. As we want to stay in the geometric context, in light of \cref{Proposition:normal_sheaf_of_morphisms_of_artin_stacks_is_again_artin} we should introduce some finiteness conditions. To avoid repeating them over and over, let us make the following convention.

\begin{Convention}
A \emph{relative Artin stack} is a morphism $\euX \rightarrow \euY$ of Artin stacks which is locally of finite type. We denote the $\infty$-category of relative Artin stacks and commutative squares by $\relativeartinstacks := \Fun_{\textnormal{loc.f.t.}}(\Delta^{1}, \artinstacks)$.
\end{Convention}
 
As a minor warning, note that the above notion of a relative Artin stack is more strict than the most general notion of a relative Artin stack in two different ways - we require the target to also be Artin, rather than arbitrary, and we require the morphism to be locally of finite type. For most applications, the stacks considered are finite type over a field, so that these two assumptions are trivially satisfied.

\begin{Notation}
If $\euX^{\prime} \rightarrow \euY^{\prime}$ and $\euX \rightarrow \euY$ are relative Artin stacks, then we will use the notation $(\euX^{\prime} \rightarrow \euY^{\prime}) \rightarrow (\euX \rightarrow \euY)$ to denote morphisms in the $\infty$-category of relative Artin stacks, which are given by commutative squares 

 \begin{center}
	\begin{tikzpicture}
		\node (TL) at (0, 1) {$\euX^{\prime} $};
		\node (TR) at (1, 1) {$\euY^{\prime} $};
		\node (BL) at (0 , 0) {$\euX$};
		\node (BR) at (1, 0){$\euY$};
		
		\draw[->] (TL) -- (TR);
		\draw[->] (BL) -- (BR);
		\draw[->] (TL) -- (BL);
		\draw[->] (TR) -- (BR);
	\end{tikzpicture}.
\end{center}
This notation is introduced to lessen our need to draw complicated diagrams.  
\end{Notation}

 \begin{Definition}
 \label{Definition:smooth_and_surjective_maps_of_relative_artin_stacks}
 We say a morphism $(\euX^{\prime} \rightarrow \euY^{\prime}) \rightarrow (\euX \rightarrow \euY)$ of relative Artin stacks is \emph{smooth} if both $\euX^{\prime} \rightarrow \euX$ and $\euY^{\prime} \rightarrow \euY$ are smooth. Likewise, we say it is \emph{surjective} if both of those arrows are surjective. 
 \end{Definition}
Since one easily verifies that the $\infty$-category of Artin stacks has universal coproducts, the same is true for the $\infty$-category of relative Artin stacks, where limits and colimits are computed separately in the source and target. It follows that $\relativeartinstacks$ admits a unique Grothendieck topology in which covering families are given by smooth, jointly surjective maps in the sense of  \cref{Definition:smooth_and_surjective_maps_of_relative_artin_stacks}. We will use descent with respect to this topology to prove the following result. 

\begin{Theorem}
\label{Theorem:unique_characterization_of_a_normal_sheaf_of_morphisms_of_artin_stacks}
The normal sheaf functor $\Ncat: \relativeartinstacks \rightarrow \artinstacks$ is determined up to a canonical natural equivalence as the unique functor subject to the following four axioms:
\begin{enumerate}
\item If $U \hookrightarrow V$ is a closed embedding of schemes, then $\Ncat_{U}V$ coincides with the classical normal sheaf, that is, $\Ncat_{U}V \simeq \mathbb{V}_{U}(I / I^{2})$, where $I$ is the ideal sheaf.
\item $\Ncat$ preserves coproducts; that is, for any two relative Artin stacks $\euX \rightarrow \euY$ and $\euX^{\prime} \rightarrow \euY^{\prime}$ we have $\Ncat_{\euX \sqcup \euX^{\prime}} (\euY \sqcup \euY^{\prime}) \simeq \Ncat_{\euX} \euY \sqcup \Ncat_{\euX^{\prime}} \euY^{\prime}$.
\item $\Ncat$ preserves smooth and smoothly surjective maps; that is, if $(\euX^{\prime} \rightarrow \euY^{\prime}) \rightarrow (\euX \rightarrow \euY)$ is smooth (resp. smooth surjective) map of relative Artin stacks, then the same is true for $\Ncat_{\euX^{\prime}} \euY^{\prime} \rightarrow \Ncat_{\euX} \euY$.
\item $\Ncat$ commutes with pullbacks along smooth morphisms; that is, if $(\euX^{\prime} \rightarrow \euY^{\prime}) \rightarrow (\euX \rightarrow \euY)$ is smooth, then $\Ncat _{\euX^{\prime} \times _{\euX} \euW} (\euY^{\prime} \times _{\euY} \euZ) \simeq \Ncat_{\euX^{\prime}} \euY^{\prime} \times _{\Ncat_{\euX} \euY} \Ncat_{\euW} \euZ$ for any $(\euW \rightarrow \euZ) \rightarrow (\euX \rightarrow \euY)$.
\end{enumerate}
\end{Theorem}

Before proving uniqueness, we will first establish that the normal sheaf does have the required properties. 

\begin{Lemma}
\label{Lemma:smooth_base_change_for_normal_sheaf}
Let $\euX \rightarrow \euY$ be a relative Artin stack and suppose that $\euY^{\prime} \rightarrow \euY$ is smooth. Then, $\Ncat_{\euX^{\prime}} \euY^{\prime} \simeq \p^{*} \Ncat_{\euX} \euY$, where $\euX^{\prime} \simeq \euX \times_{\euY} \euY^{\prime}$ and $p: \euX^{\prime} \rightarrow \euX$ is the projection. 
\end{Lemma}

\begin{proof}
By flat base-change for the cotangent complex, we have $L_{\euX^{\prime} / \euY^{\prime}} \simeq p^{*} L_{\euX / \euY}$, so that the statement follows immediately from \cref{Lemma:base_change_for_cone_functor}. 
\end{proof}

\begin{Lemma}
\label{Lemma:cofibre_sequences_of_abelian_stacks_involving_the_normal_sheaf}
Suppose we have a composite $\euX \rightarrow \euY \rightarrow \euZ$ of morphisms of Artin stacks. Then

\begin{enumerate}
\item if $\euX \rightarrow \euY$ is smooth, there's a cofibre sequence $\Ncat_{\euX} \euY \rightarrow \Ncat_{\euX} \euZ \rightarrow \euX \times _{\euY} \Ncat_{\euY} \euZ$
\item if $\euY \rightarrow \euZ$ is smooth, there's a cofibre sequence $\Omega_{\euX} (\euX \times _{\euY} \Ncat_{\euY} \euZ) \rightarrow \Ncat_{\euX} \euY \rightarrow \Ncat_{\euX} \euZ$
\end{enumerate} 
of abelian stacks over $\euX$, where $\Omega _{\euX}$ is the fibrewise loop space over $\euX$. Moreover, in both cases the left term is smooth over $\euX$ and so the right map is a smooth surjection. 
\end{Lemma}

\begin{proof}
Using the standard exactness properties of the cotangent complex, the first cofibre sequence follows from \cref{Remark:map_of_qcoh_sheaves_with_coconnective_perfect_cofibre_induces_cofibre_seq_of_abelian_stacks} applied to $p^{*} L_{\euY / \euZ}[-1] \rightarrow L_{\euX / \euZ}[-1] \rightarrow L_{\euX / \euY}[-1]$, where $p: \euX \rightarrow \euY$, and the second from applying it to $L_{\euX / \euZ}[-1] \rightarrow L_{\euX / \euY} \rightarrow p^{*} L_{\euY / \euZ}$.

To see the second claim, observe that in a cofibre sequence, the right map is always surjective, and smoothness follow from \cref{Theorem:cone_over_tor_bounded_qcoh_sheaf_is_artin_and_smooth_for_perfects}.
\end{proof}

\begin{Lemma}
\label{Lemma:normal_sheaf_functor_preserves_smooth_and_smoothly_surjective_maps_of_relative_artin_stacks}
The normal sheaf functor $\Ncat: \relativeartinstacks \rightarrow \artinstacks$ preserves smooth and smoothly surjective maps. 
\end{Lemma}

\begin{proof}
Any morphism $(\euX^{\prime} \rightarrow \euY^{\prime}) \rightarrow (\euX \rightarrow \euY)$ of relative Artin stacks can be decomposed as 

 \begin{center}
	\begin{tikzpicture}
		\node (TL) at (0, 1) {$\euX^{\prime} $};
		\node (TR) at (1, 1) {$\euY^{\prime} $};
		\node (BL) at (0 , 0) {$\euX^{\prime}$};
		\node (BR) at (1, 0){$\euY$};
		\node (BBL) at (0 , -1) {$\euX$};
		\node (BBR) at (1, -1){$\euY$};
		
		\draw[->] (TL) -- (TR);
		\draw[->] (BL) -- (BR);
		\draw[->] (TL) -- (BL);
		\draw[->] (TR) -- (BR);
		\draw[->] (BL) -- (BBL);
		\draw[->] (BR) -- (BBR);
		\draw[->] (BBL) -- (BBR);
	\end{tikzpicture};
\end{center}
that is, into a composite of two other morphisms for which either the map on the source or on the target is the identity. If the given morphism is smooth (resp. smooth and surjective), so are the two factors, so that we can reduce to this case. 

The fact that $\Ncat _{\euX^{\prime}} \euY^{\prime} \rightarrow \Ncat _{\euX^{\prime}} \euY$ is smooth and surjective is immediate from the second cofibre sequence of \cref{Lemma:cofibre_sequences_of_abelian_stacks_involving_the_normal_sheaf}. The first part of the same result implies that $\Ncat _{\euX^{\prime}} \euY \rightarrow \euX^{\prime} \times _{\euX} \Ncat_{\euX} \euY$ is smooth and surjective, and so the observation that $\euX^{\prime} \times _{\euX} \Ncat_{\euX} \euY \rightarrow \Ncat_{\euX} \euY$ is smooth (resp. smooth and surjective) whenever $\euX^{\prime} \rightarrow \euX$ is finishes the claim. 
\end{proof}

\begin{Proposition}
\label{Proposition:normal_sheaf_commutes_with_pullbacks_along_smooth_maps_relative_artin_stacks}
Let $(\euX^{\prime} \rightarrow \euY^{\prime}) \rightarrow (\euX \rightarrow \euY)$ be a smooth map of relative Artin stacks and let $(\euW \rightarrow \euZ) \rightarrow (\euX \rightarrow \euY)$ be arbitrary. Then, $\Ncat _{\euX^{\prime} \times _{\euX} \euW} (\euY^{\prime} \times _{\euY} \euZ) \simeq \Ncat_{\euX^{\prime}} \euY^{\prime} \times _{\Ncat_{\euX} \euY} \Ncat_{\euW} \euZ$ 
\end{Proposition}

\begin{proof}
Using the decomposition of a smooth morphism as in the proof of \cref{Lemma:normal_sheaf_functor_preserves_smooth_and_smoothly_surjective_maps_of_relative_artin_stacks} we can assume that the given smooth map is the identity in either the source or target. 

Let us tackle first the case when $\euY^{\prime} = \euY$. Using the notation $\euW^{\prime} := \euW \times _{\euX} \euX^{\prime}$, our goal is to show that $\Ncat_{\euW^{\prime}} \euZ \simeq \Ncat_{\euW} \euZ \times _{\Ncat _{\euX} \euY} \Ncat _{\euX^{\prime}} \euY$. By \cref{Lemma:cofibre_sequences_of_abelian_stacks_involving_the_normal_sheaf}, we have a cofibre sequence

\begin{center}
$\Ncat_{\euX^{\prime}} \euX \rightarrow \Ncat_{\euX^{\prime}} \euY \rightarrow \euX^{\prime} \times _{\euX} \Ncat_{\euX} \euY$
\end{center}
of abelian stacks over $\euX^{\prime}$. Said differently, the augmented simplicial object 

\begin{center}
$\ldots \triplerightarrow \Ncat_{\euX^{\prime}} \euX \times_{\euX^{\prime}} \Ncat_{\euX^{\prime}} \euY \rightrightarrows \Ncat_{\euX^{\prime}} \euY \rightarrow \euX^{\prime} \times _{\euX} \Ncat_{\euX} \euY$
\end{center}
determined by the action is a colimit diagram. By direct inspection, applying $- \times _{\Ncat_{\euX} \euY} \Ncat_{\euW} \euZ$ to this diagram yields 

\begin{center}
$\ldots \triplerightarrow \Ncat_{\euW^{\prime}} \euW \times_{\euW^{\prime}} (\Ncat_{\euX^{\prime}} \euY \times _{\Ncat_{\euX} \euY} \Ncat_{\euW} \euZ) \rightrightarrows \Ncat_{\euX^{\prime}} \euY \times _{\Ncat_{\euX} \euY} \Ncat_{\euW} \euZ \rightarrow \euW^{\prime} \times _{\euW} \Ncat_{\euW} \euZ$,
\end{center}
where we've used that the base-change formula $\Ncat_{\euW^{\prime}} \euW \simeq \euW^{\prime} \times _{\euX^{\prime}} \Ncat_{\euX^{\prime}} \euX$ of \cref{Lemma:smooth_base_change_for_normal_sheaf}. Since taking pullbacks in stacks preserves colimits, this presents a cofibre sequence 

\begin{center}
$\Ncat_{\euW^{\prime}} \euW \rightarrow \Ncat_{\euX^{\prime}} \euY \times _{\Ncat_{\euX} \euY} \Ncat_{\euW} \euZ \rightarrow \euW^{\prime} \times _{\euW} \Ncat_{\euW} \euZ$
\end{center}
of abelian stacks on $\euW^{\prime}$. Since $\Ncat_{\euW^{\prime}} \euZ$ is also a middle term of such a cofibre sequence by the same argument, and these cofibre sequences are natural, we deduce that $\Ncat_{\euW^{\prime}} \euZ \simeq \Ncat_{\euX^{\prime}} \euY \times _{\Ncat_{\euX} \euY} \Ncat_{\euW} \euZ$, which is what we wanted to show. 

Let us now suppose that $\euX^{\prime} = \euX$, our goal is to show that $\Ncat_{\euX} \euY^{\prime} \times _{\Ncat_{\euX} \euY} \Ncat _{\euW} \euZ \simeq \Ncat_{\euW} \euZ^{\prime}$, where $\euZ^{\prime} := \euZ \times _{\euY} \euY^{\prime}$. Using \cref{Lemma:cofibre_sequences_of_abelian_stacks_involving_the_normal_sheaf} again we have a cofibre sequence

\begin{center}
$\Omega _{\euX} (\euX \times _{\euY^{\prime}} \Ncat_{\euY^{\prime}} \euY) \rightarrow \Ncat _{\euX} \euY^{\prime} \rightarrow \Ncat _{\euX} \euY$.
\end{center}
and the same argument as before shows that by applying $- \times _{\Ncat _{\euX} \euY} \Ncat _{\euW} \euZ$ we get a cofibre sequence

\begin{center}
$\Omega _{\euW} (\euW \times _{\euZ^{\prime}} \Ncat_{\euZ^{\prime}} \euZ) \rightarrow \Ncat _{\euX} \euY^{\prime} \times _{\Ncat _{\euX} \euY} \Ncat _{\euW} \euZ \rightarrow \Ncat_{\euW} \euZ$.
\end{center}
Since $\Ncat_{\euW} \euZ^{\prime}$ is also a middle term of a cofibre sequence of this form, this ends the argument. 
\end{proof}
  
Observe that \cref{Example:normal_sheaf_functor_on_closed_embedding_of_affines_is_same_as_classical}, \cref{Lemma:normal_sheaf_functor_preserves_smooth_and_smoothly_surjective_maps_of_relative_artin_stacks} and \cref{Proposition:normal_sheaf_commutes_with_pullbacks_along_smooth_maps_relative_artin_stacks} taken together already verify all of the properties of the normal sheaf functor spelled out in the statement of \cref{Theorem:unique_characterization_of_a_normal_sheaf_of_morphisms_of_artin_stacks}. Thus, to complete the proof of the latter, we only have to check that $\Ncat$ is the \emph{unique} functor subject to these conditions.

Note that the only part of \cref{Theorem:unique_characterization_of_a_normal_sheaf_of_morphisms_of_artin_stacks} that specifies values of the normal sheaf without reference to anything else, is the property that $\Ncat_{U} V \simeq \mathbb{V}_{U}(I / I^{2})$ for a closed embedding of schemes with ideal sheaf $I$. As this context will occur frequently, let us introduce an appropriate terminology. 

\begin{Definition}
A \emph{pair} is a closed embedding $U \hookrightarrow V$ of schemes. The category $\pairs$ of pairs is a full subcategory of the $\infty$-category of relative Artin stacks. 
\end{Definition} 
In this language, we need to show that the normal sheaf functor is determined by its interaction with smooth maps together with its values at pairs of schemes. To show that the latter alone suffices, we will use the topology on $\relativeartinstacks$ determined by smooth surjections in the sense of \cref{Definition:smooth_and_surjective_maps_of_relative_artin_stacks}. The key is the following slightly surprising fact. 

\begin{Lemma}
\label{Lemma:closed_embedding_between_affines_generate_all_artin_stacks}
Any relative Artin stack admits a smooth surjection from a pair of schemes. In fact, the category $\pairs$ is a generating subcategory of $\relativeartinstacks$ in the sense of \cref{Definition:downward_closed_and_generating_subcategory}.
\end{Lemma}

\begin{proof}
Since any $n$-Artin stack admits a smooth cover from a disjoint union of affine schemes whose iterated intersections are $(n-1)$-Artin stacks, it is easy to see by induction that the $\infty$-category of relative Artin stacks is generated by morphisms between such schemes. Thus, it is enough to check that the latter is generated by closed embeddings. 

Let $f: X \rightarrow Y$ be a relative Artin stack where $X, Y$ are disjoint unions of affine schemes, and let us choose one such decomposition $X \simeq \bigsqcup X_{\alpha}$. Since each $X_{\alpha}$ is quasi-compact, its image is contained in a finite disjoint union of affines, and hence an affine disjoint summand which we will denote by $Y_{\alpha} \subseteq Y$. 

Let us write $X_{\alpha} \simeq \spec(B_{\alpha})$, $Y_{\alpha} \simeq \spec(A_{\alpha})$. By assumption that $f$ is locally of finite type, $A_{\alpha} \rightarrow B_{\alpha}$ is finitely generated, and hence we can find a factorization 
\[
A_{\alpha} \hookrightarrow C_{\alpha} \twoheadrightarrow B_{\alpha},
\]
where $C_{\alpha} := A_{\alpha}[x_{1}, \ldots, x_{n}]$ for some $n$. It follows that if we write $Z_{\alpha} := \spec(C_{\alpha})$, then in the induced factorization 
\[
X_{\alpha} \hookrightarrow Z_{\alpha} \rightarrow Y_{\alpha}
\]
the first arrow is a closed inclusion of schemes and the second is smooth. Writing $Z := \bigsqcup Z_{\alpha}$, the same is true for 
\[
X  \rightarrow Z \rightarrow Y
\]
If we now let $Y^{\prime} := Z \sqcup Y$, with the induced map from $X$ the composite $Z \rightarrow Z \rightarrow Y^{\prime}$, then in the factorization
\[
X \rightarrow Y^{\prime} \rightarrow Y
\]
the first arrow is again a closed inclusion and the second one is now a smooth surjection.

The above factorization determines a smooth surjection $(X \hookrightarrow Y^{\prime}) \rightarrow (X \rightarrow Y)$ of relative Artin stacks whose source is a pair. However, the same is true for all of the iterated intersections, as they are of the form $X \hookrightarrow Y^{\prime} \times _{Y} \ldots \times_{Y} Y^{\prime}$ and these are easily seen to be closed embeddings of schemes. We deduce that $(X \rightarrow Y)$ is in the subcategory generated by pairs, ending the argument. 
\end{proof}

We are now ready to prove the main result of this section. 

\begin{proof}[Proof of Theorem \ref{Theorem:unique_characterization_of_a_normal_sheaf_of_morphisms_of_artin_stacks}]
We've already verified all of the requires properties of the normal sheaf functor $\Ncat: \relativeartinstacks \rightarrow \artinstacks$, all that is left is uniqueness. As a consequence of \cref{Proposition:adapted_cosheaves_are_functors_that_preserve_pullbacks_and_surjections}, any functor satisfying these properties is a cosheaf with respect to the smooth topology on Artin stacks. Since by \cref{Lemma:closed_embedding_between_affines_generate_all_artin_stacks} the subcategory of pairs of schemes is generating, \cref{Proposition:restriction_of_sheaves_to_generating_subcategory_is_an_equivalence} implies that any such cosheaf is uniquely determined by its restriction to $\pairs$, ending the proof. 
\end{proof}

\section{The normal cone of a morphism of Artin stacks}

In this section we generalize the construction of the normal cone of a closed embedding of schemes to any locally of finite type morphism of Artin stacks. We characterize our extension as the unique one satisfying certain natural axioms and verify that in the case of a morphism of Deligne-Mumford type, our construction agrees with that of Behrend and Fantechi. 

Recall that if $U \hookrightarrow V$ is a closed embedding of schemes with ideal sheaf $I$, then the \emph{normal cone} $\Ccat_{U} V := \spec_{U}(\bigoplus I^{k} / I^{k+1})$ is the relative spectrum of the associated graded $\mathcal{O}_{V}$-algebra. The construction of the normal cone is fundamental in intersection theory \cite{fulton2013intersection}.

\begin{Notation}
We will use the symbol $\Ccat$ to denote the normal cone of a morphism, rather than $\mathbb{V}$, which we reserve for the abelian cone associated to a quasi-coherent sheaf introduced in \cref{Definition:abelian_cone_associated_to_a_qcoh_sheaf}. The normal cone is usually not abelian. 
\end{Notation}

The normal cone is intimately related to the normal sheaf discussed in the previous chapter; the graded algebra $\bigoplus I^{k} / I^{k+1}$ is generated in degree $1$, and it follows that there is a canonical closed embedding $\Ccat_{U} V \hookrightarrow \Ncat_{U} V$ into the normal sheaf. If $U \hookrightarrow V$ is regular, this embedding is an isomorphism, and so one can consider the normal cone as a measure of non-smoothness. 

In \cref{Theorem:unique_characterization_of_a_normal_sheaf_of_morphisms_of_artin_stacks} we proved that the natural extension of the notion of a normal sheaf using the theory of the cotangent complex can be characterized uniquely by simple axioms. We will now prove that an analogous extension can also be constructed for the normal cone.

\begin{Theorem}
\label{Theorem:there_exists_a_unique_normal_cone_functor_subject_to_properties}
There exists a unique functor $\Ccat: \relativeartinstacks \rightarrow \artinstacks$ on the $\infty$-category of relative Artin stacks, called the \emph{normal cone}, such that: 
\begin{enumerate}
\item if $U \hookrightarrow V$ is a closed embedding of schemes, then $\Ccat_{U}V$ coincides with the classical normal cone, that is, $\Ccat_{U}V \simeq \spec_{U}(\bigoplus I^{k} / I^{k+1})$, where $I$ is the ideal sheaf
\item $\Ccat$ preserves coproducts
\item $\Ccat$ preserves smooth and smoothly surjective maps
\item $\Ccat$ commutes with pullbacks along smooth morphisms of relative Artin stacks
\end{enumerate}
\end{Theorem}
Note that the content of the above result is slightly different than that of \cref{Theorem:unique_characterization_of_a_normal_sheaf_of_morphisms_of_artin_stacks} which concerned the normal sheaf, as in the latter case we constructed the functor a priori. In the case of the normal cone, the existence of this functor is part of the result. Nevertheless, the normal cone is strongly related to the normal sheaf, as the following shows.

\begin{Theorem}
\label{Theorem:normal_cone_is_a_closed_substack_of_normal_sheaf_and_has_cartesian_property_for_smooth_morphisms}
There is a unique natural transformation $\Ccat \rightarrow \Ncat$ of functors on relative Artin stacks which for every pair $U \hookrightarrow V$ of schemes with ideal sheaf $I$ coincides with the natural morphism $\Ccat_{U} V \rightarrow \Ncat_{U} V$ induced by the surjection $\textnormal{Sym} (I / I^{2}) \rightarrow \bigoplus I^{k} / I^{k+1}$. 

Moreover, for any relative Artin stack $\euX \rightarrow \euY$

\begin{enumerate}
\item $\Ccat_{\euX} \euY \rightarrow \Ncat_{\euX} \euY$ is a closed embedding and
\item for any smooth morphism $(\widetilde{\euX} \rightarrow \widetilde{\euY}) \rightarrow (\euX \rightarrow \euY)$ of relative Artin stacks the induced square

 \begin{center}
	\begin{tikzpicture}
		\node (TL) at (0, 1.5) {$ \Ccat_{\widetilde{\euX}} \widetilde{\euY} $};
		\node (TR) at (1.5, 1.5) {$ \Ncat_{\widetilde{\euX}} \widetilde{\euY} $};
		\node (BL) at (0 , 0) {$ \Ccat_{\euX} \euY $};
		\node (BR) at (1.5, 0){$ \Ncat_{\euX} \euY $};
		
		\draw[->] (TL) -- (TR);
		\draw[->] (BL) -- (BR);
		\draw[->] (TL) -- (BL);
		\draw[->] (TR) -- (BR);
	\end{tikzpicture}
\end{center}
is cartesian.
\end{enumerate}
\end{Theorem}
The proofs of these two results are intimately related. Observe that \cref{Theorem:there_exists_a_unique_normal_cone_functor_subject_to_properties} implies that $\Ccat$ is a cosheaf adapted to the class of smooth maps in the sense of \cref{Definition:adapted_cosheaf}. Since we've shown before in \cref{Theorem:extension_of_an_adapted_cosheaf_is_adapted} that an adapted cosheaf can be uniquely extended from a generating subcategory, it is enough to verify that the classical normal cone of a closed embedding of schemes has the required properties. 

The latter is a problem in commutative algebra which can be tackled directly, but we will not do so. Instead, we verify that the cartesian property of \cref{Theorem:normal_cone_is_a_closed_substack_of_normal_sheaf_and_has_cartesian_property_for_smooth_morphisms} holds for smooth morphisms between pairs of schemes, the other  properties will then follow from what we've already proven about the normal sheaf. 

Before proceeding with the proofs, let us first show that expected properties of the normal cone follow from the above axiomatics. 

\begin{Lemma}
For any $\euX \rightarrow \euY$, the normal cone $\Ccat_{\euX} \euY$ is canonically a pointed stack over $\euX$. 
\end{Lemma}

\begin{proof}
Observe that both $\Ccat$ and the "source" functor $s(\euX \rightarrow \euY) = \euX$ are adapted cosheaves on relative Artin stacks. It follows that any natural transformations between them defined for closed embeddings of affine schemes extend uniquely to all of $\relativeartinstacks$. In particular, for any $\euX \rightarrow \euY$ there's a canonical projection $\Ccat_{\euX} \euY \rightarrow \euX$, and this projection admits a canonical section, because that is the case in the classical setting. 
\end{proof}

\begin{Remark}
The classical normal cone $C_{X}Y$ of a closed embedding $X \hookrightarrow Y$ of schemes has more structure than just being pointed, namely, it is also equipped with an action of $\mathbb{A}^{1}$ induced by the grading of $\bigoplus I^{k} / I^{k+1}$. 

A slight variation in the arguments we give shows that the same is true for the normal cone for an arbitrary morphisms of Artin stacks. Namely, instead of considering $\Ccat$ as an adapted cosheaf valued in Artin stacks, one should consider it as valued in the $\infty$-category of morphisms $\EuScript{M} \rightarrow \euX$ where $\EuScript{M}$ is a pointed $\euX$-Artin stack equipped with an $\mathbb{A}^{1}$-action.
\end{Remark}

\begin{Lemma}
\label{Lemma:normal_cone_of_identity_of_an_artin_stack_is_trivial}
The normal cone of the identity is trivial; that is, for any $\euX$ we have $\Ccat_{\euX} \euX \simeq \euX$. 
\end{Lemma}

\begin{proof}
Since $\Ccat$ is a cosheaf on relative Artin stacks, it is easy to see that the subcategory of those Artin stacks which satisfy the above condition is downward closed. Since it contains all affine schemes, we deduce that it must be all of $\artinstacks$. 
\end{proof}

\begin{Proposition}[Smooth base-change]
\label{Proposition:smooth_base_change}
The normal cone satisfies smooth base-change. That is, for any $\euX \rightarrow \euY$ and smooth $\euY^{\prime} \rightarrow \euY$ we have $\Ccat_{\euX^{\prime}} \euY^{\prime} \simeq \euX^{\prime} \times _{\euX} \Ccat_{\euX} \euY$, where $\euX^{\prime} \simeq \euY^{\prime} \times _{\euY} \euX$. 
\end{Proposition}

\begin{proof}
Keeping in mind \cref{Lemma:normal_cone_of_identity_of_an_artin_stack_is_trivial}, this is immediate from applying the pullback axiom to the span $(\euX \rightarrow \euY) \rightarrow (\euY \rightarrow \euY) \leftarrow (\euY^{\prime} \rightarrow \euY^{\prime})$.
\end{proof}

\begin{Proposition}[\'{E}tale invariance]
\label{Proposition:etale_invariance_of_the_normal_cone_for_relative_artin_stacks}
Suppose we have a morphism $(\widetilde{\euX} \rightarrow \widetilde{\euY}) \rightarrow (\euX \rightarrow \euY)$ of relative Artin stacks which is \'{e}tale on both source and target. Then, $\Ccat_{\widetilde{\euX}} \widetilde{\euY} \simeq \widetilde{\euX} \times _{\euX} \Ccat _{\euX} \euY$. 
\end{Proposition}

\begin{proof}
Since $(\widetilde{\euX} \rightarrow \widetilde{\euY}) \rightarrow (\euX \rightarrow \euY)$ is smooth, this is immediate from the cartesian square of \cref{Theorem:normal_cone_is_a_closed_substack_of_normal_sheaf_and_has_cartesian_property_for_smooth_morphisms} and the \'{e}tale invariance of the normal sheaf. 
\end{proof}
We can also verify that in the Deligne-Mumford case, our construction recovers the intrinsic normal cone of Behrend and Fantechi. 

\begin{Theorem}
\label{Theorem:comparison_with_Behrend_Fantechi_normal_cone}
Let $\euX \rightarrow \euY$ be a morphism of $1$-Artin stacks of finite type which is a relative Deligne-Mumford stack. Then, the normal cone $\Ccat_{\euX} \euY$ is equivalent to the intrinsic normal cone of Behrend-Fantechi. 
\end{Theorem}

\begin{proof}
Since both the Behrend-Fantechi intrinsic normal cone and the normal cone of \cref{Theorem:there_exists_a_unique_normal_cone_functor_subject_to_properties} satisfy smooth base-change and are \'{e}tale-invariant, the latter by \cref{Proposition:smooth_base_change} and \cref{Proposition:etale_invariance_of_the_normal_cone_for_relative_artin_stacks}, we can assume that we have a morphism $X \rightarrow Y$ of schemes. 

We can lift the given morphism to a closed embedding $i: X \hookrightarrow M$ such that $M \rightarrow Y$ is smooth and surjective. Observe that this then defines a smooth surjection 

\begin{center}
$(X \hookrightarrow M) \rightarrow (X \rightarrow Y)$
\end{center}
of relative Artin stacks. It follows from \cref{Theorem:normal_cone_is_a_closed_substack_of_normal_sheaf_and_has_cartesian_property_for_smooth_morphisms} that we have a cartesian diagram

 \begin{center}
	\begin{tikzpicture}
		\node (TL) at (0, 1.5) {$ \Ccat_{X} M $};
		\node (TR) at (2, 1.5) {$ \Ccat_{X} Y $};
		\node (BL) at (0 , 0) {$ \Ncat_{X} M $};
		\node (BR) at (2, 0){$ \Ncat_{X} Y $};
		
		\draw[->] (TL) -- (TR);
		\draw[->] (BL) -- (BR);
		\draw[->] (TL) -- (BL);
		\draw[->] (TR) -- (BR);
	\end{tikzpicture},
\end{center}
which is precisely how Behrend-Fantechi defined the intrinsic normal cone \cite{behrend_fantechi_intrinsic_normal_cone}[3.10].
\end{proof}

\begin{Remark}
Note that even in the classical Deligne-Mumford case, \cref{Theorem:there_exists_a_unique_normal_cone_functor_subject_to_properties} clarifies the construction of Behrend and Fantechi by showing that it is the only extension of the normal cone of a closed embedding of schemes that preserves certain natural properties.
\end{Remark}

The rest of this chapter will be devoted to the proofs of \cref{Theorem:there_exists_a_unique_normal_cone_functor_subject_to_properties} and \cref{Theorem:normal_cone_is_a_closed_substack_of_normal_sheaf_and_has_cartesian_property_for_smooth_morphisms}; as explained above, the main step is to establish that the classical normal cone has the required properties when restricted to the category of pairs $U \hookrightarrow V$ of schemes.

In more detail, we need to prove that the normal cone functor $\Ccat: \pairs \rightarrow \artinstacks$ preserves smooth and smoothly surjective morphisms, and commutes with pullbacks along smooth morphisms. To do so, it will be convenient to introduce some temporary terminology. 

\begin{Definition}
\label{Definition:good_morphism_of_pairs_of_schemes}
We say a morphism $f: (N \hookrightarrow M) \rightarrow (X \hookrightarrow Y)$ of pairs is \emph{good} if the induced diagram 

 \begin{center}
	\begin{tikzpicture}
		\node (TL) at (0, 1.5) {$ \Ccat_{N} M $};
		\node (TR) at (1.5, 1.5) {$ \Ncat_{N} M $};
		\node (BL) at (0 , 0) {$ \Ccat_{X} Y $};
		\node (BR) at (1.5, 0){$ \Ncat_{X} Y$};
		
		\draw[->] (TL) -- (TR);
		\draw[->] (BL) -- (BR);
		\draw[->] (TL) -- (BL);
		\draw[->] (TR) -- (BR);
	\end{tikzpicture}
\end{center}
between the normal cones and normal sheaves is cartesian.
\end{Definition}
Our goal is to prove that an arbitrary smooth morphisms of pairs; that is, one that is smooth on both source and target, is good. As an easy example, any flat cartesian morphism of pairs is good, as an easy consequence of the flat base-change for the normal cone and the normal sheaf. 

\begin{Remark}
It is not true that every morphism of pairs of schemes is good in the sense of \cref{Definition:good_morphism_of_pairs_of_schemes}. As an example, let $L \subseteq \mathbb{A}^{2}$ be the union of the coordinate axes and let $\spec(k) \hookrightarrow L$ be the inclusion of the origin. Then, one can verify that the obvious morphism $(\spec(k) \hookrightarrow L) \rightarrow (\spec(k) \rightarrow \mathbb{A}^{2})$ induces an isomorphism between normal bundles, but not between the normal cones. 
\end{Remark}

\begin{Lemma}
\label{Lemma:any_etale_morphism_of_pairs_is_good}
Any morphism $(\widetilde{X} \hookrightarrow \widetilde{Y}) \rightarrow (X \hookrightarrow Y)$ of pairs of schemes which is \'{e}tale on both source and target is good. 
\end{Lemma} 

\begin{proof}
Since both the normal sheaf and normal cone satisfy flat base-change, we can replace $Y$ by the spectrum of the strict henselization of the local ring at each of its points; then, both $X$ and $Y$ will be of this form. It follows that $\widetilde{X}$ is a disjoint union of copies of $X$ mapping onto it isomorphically, and likewise for $\widetilde{Y}$. The claim then follows.
\end{proof}

\begin{Lemma}
\label{Lemma:the_property_of_a_morphism_of_pairs_being_good_is_local}
Suppose we have a morphism $f: (N \hookrightarrow M) \rightarrow (X \hookrightarrow Y)$ of pairs of schemes. Then, $f$ is good if and only if for each $n \in N$ there exists an open neighbourhood $U \subseteq M$ of $n$  such that the restriction $(U \cap N \hookrightarrow U) \rightarrow (X \hookrightarrow Y)$ is good. 
\end{Lemma}

\begin{proof}
Observe that for any such neighbourhood we have $\Ccat_{U \cap N} U \simeq \Ccat_{N} M \times _{N} (U \cap N)$ and likewise for the normal sheaf, and that as $n \in N$ varies the open sets $U \cap N$ form a covering of $N$. Then, since $\Ccat_{N} M \rightarrow \Ccat_{X} Y \times _{\Ncat_{X} Y} \Ncat _{N} M$ is a morphism of $N$-schemes, the claim is then equivalent to saying that is in isomorphism if and only if the same is true for the maps 

\begin{center}
$\Ccat_{U \cap N} M \simeq \Ccat_{N} M \times _{N} (U \cap N) \rightarrow \Ccat_{X} Y \times _{\Ncat_{X} Y} \Ncat _{N} M \times _{N} (U \cap N) \simeq \Ccat_{X} Y \times _{\Ncat_{X}Y} \Ncat_{U \cap N} U$.
 \end{center}
\end{proof}

\begin{Lemma}
\label{Lemma:change_in_normal_cone_when_composing_with_inclusion_into_affine_space}
Suppose we have a morphism of pairs of schemes of the form 

 \begin{center}
	\begin{tikzpicture}
		\node (TL) at (0, 1.2) {$ X $};
		\node (TR) at (1.5, 1.2) {$ Y \times \mathbb{A}^{n} $};
		\node (BL) at (0 , 0) {$ X  $};
		\node (BR) at (1.5, 0){$ Y $};
		
		\draw[->] (TR) -- (BR);
		\draw[->] (TL) -- (BL);
		\draw[right hook->] (BL) -- (BR);
		\draw[right hook->] (TL) -- (TR);
	\end{tikzpicture},
\end{center}
where the left vertical arrow is the identity, the right one is the projection, and the upper horizontal arrow is the composite $X \hookrightarrow Y \times \{ 0 \} \hookrightarrow Y \times \mathbb{A}^{n}$ of the lower one with the natural inclusion. Then, any such morphism is good. 
\end{Lemma}

\begin{proof}
In this case, one can compute directly that $\Ccat_{X} (Y \times \mathbb{A}^{n}) \simeq (\Ccat_{X} Y) \times \mathbb{A}^{n}$, see \cite{battistella2018virtual_classes_for_working_mathematician}[3.5.1], and since an analogous formula holds for the normal sheaf, the claim follows. 
\end{proof}

\begin{Proposition}
\label{Proposition:any_smooth_morphism_of_pairs_is_good_so_that_it_induces_cartesian_diagram_between_normal_cone_and_sheaf}
Any smooth morphism $(N \hookrightarrow M) \rightarrow (X \hookrightarrow Y)$ of pairs of schemes is good. In other words, for any such morphism we have $\Ccat_{N} M \simeq \Ccat_{X} Y \times_{\Ncat_{X}Y} \Ncat_{N}M$.
\end{Proposition} 

\begin{proof}
Observe that the induced morphism $N \hookrightarrow X \times _{Y} M$ is a closed embedding of smooth $X$-schemes. It follows that by picking $n \in N$ and choosing a smaller affine neighbourhood of its image in $M$, which we can do by \cref{Lemma:the_property_of_a_morphism_of_pairs_being_good_is_local}, we can assume that there are regular functions $(g_{i})_{1 \leq i \leq n}$ on $X \times _{Y} M$ and an $m \leq n$ such that the resulting diagram 

 \begin{center}
	\begin{tikzpicture}
		\node (TL) at (0, 1) {$ N $};
		\node (TR) at (2.5, 1) {$ X \times _{Y} M $};
		\node (BL) at (0 , 0) {$ X \times \mathbb{A}^{m} $};
		\node (BR) at (2.5, 0){$ X \times \mathbb{A}^{n} $};
		
		\draw[->] (TR) -- (BR);
		\draw[->] (TL) -- (BL);
		\draw[right hook->] (BL) -- (BR);
		\draw[right hook->] (TL) -- (TR);
	\end{tikzpicture},
\end{center}
where the bottom arrow is the natural inclusion, is cartesian and with vertical arrows \'{e}tale \cite{sga1}[Exp.II, Prop 4.9]. By lifting those regular functions to all of $M$, we can extend the right vertical arrow to a morphism $M \rightarrow Y \times \mathbb{A}^{n}$ and by making $M$ smaller if necessary we can assume that the latter is \'{e}tale as well. We can then consider the larger diagram 

 \begin{center}
	\begin{tikzpicture}
		\node (TL) at (0, 1) {$ N $};
		\node (TR) at (2.5, 1) {$ M $};
		\node (BL) at (0 , 0) {$ X \times \mathbb{A}^{m} $};
		\node (BR) at (2.5, 0){$ Y \times \mathbb{A}^{n} $};
		\node (BBL) at (0 , -1) {$ X \times \mathbb{A}^{m} $};
		\node (BBR) at (2.5, -1){$ Y \times \mathbb{A}^{m} $};
		\node (BBBL) at (0 , -2) {$ X $};
		\node (BBBR) at (2.5, -2){$ Y $};
		
		\draw[->] (TR) -- (BR);
		\draw[->] (TL) -- (BL);
		\draw[->] (BR) -- (BBR);
		\draw[->] (BL) -- (BBL);
		\draw[->] (BBR) -- (BBBR);
		\draw[->] (BBL) -- (BBBL);
		\draw[right hook->] (BBL) -- (BBR);
		\draw[right hook->] (BBBL) -- (BBBR);
		\draw[right hook->] (BL) -- (BR);
		\draw[right hook->] (TL) -- (TR);
	\end{tikzpicture},
\end{center}
where the map $Y \times \mathbb{A}^{n} \rightarrow Y \times \mathbb{A}^{m}$ is the obvious projection. Out of the three squares stacked on top of each other, the bottom one is good because it is smooth cartesian and the top one by \cref{Lemma:any_etale_morphism_of_pairs_is_good}. Since the middle square is good by \cref{Lemma:change_in_normal_cone_when_composing_with_inclusion_into_affine_space} and composition of good squares is good by the pullback pasting lemma, we are done. 
\end{proof}

We are now ready to give proofs of the two main results of this chapter. 

\begin{proof}[Proof of \cref{Theorem:there_exists_a_unique_normal_cone_functor_subject_to_properties} and \cref{Theorem:normal_cone_is_a_closed_substack_of_normal_sheaf_and_has_cartesian_property_for_smooth_morphisms}]
We first claim that $\Ccat: \pairs \rightarrow \artinstacks$ is an  cosheaf on the site of pairs of schemes adapted to the class of smooth maps. We have a natural transformation $\Ccat \rightarrow \Ncat$, which as we verified in \cref{Proposition:any_smooth_morphism_of_pairs_is_good_so_that_it_induces_cartesian_diagram_between_normal_cone_and_sheaf} is smooth-cartesian in the sense of \cref{Definition:tcartesian_transformation_between_cosheaves}. It then follows from \cref{Lemma:property_of_being_an_adapted_cosheaf_descends_along_cartesian_morphism} that $\Ccat$ is adapted, because this is true for the normal sheaf as a consequence of  \cref{Theorem:unique_characterization_of_a_normal_sheaf_of_morphisms_of_artin_stacks}. 

We thus deduce from \cref{Proposition:restriction_of_sheaves_to_generating_subcategory_is_an_equivalence} and \cref{Lemma:closed_embedding_between_affines_generate_all_artin_stacks} that the normal cone functor uniquely extends to a $\stacks$-valued cosheaf on all relative Artin stacks, and moreover that this cosheaf is also adapted by \cref{Theorem:extension_of_an_adapted_cosheaf_is_adapted}. We will now show that for any relative Artin stack $\euX \rightarrow \euY$, the stack $\Ccat_{\euX} \euY$ is in fact Artin. 

We claim that the subcategory of those relative Artin stacks for which the normal cone is Artin is downward closed, since we know it contains all pairs of schemes, this will imply the claim. Suppose that $(\euX_{0} \rightarrow \euY_{0}) \rightarrow (\euX \rightarrow \euY)$ is a smooth surjection such that $\Ccat _{\euX_{k}} \euY_{k}$ is Artin, where $\euX_{k} := \euX_{0} \times _{\euX} \ldots \times _{\euX} \euX_{0}$ and likewise for $\euY$. Since $\Ccat$ is an adapted cosheaf, we see that the diagram 

\begin{center}
$\ldots \triplerightarrow \Ccat_{\euX_{1}} \euY_{1} \rightrightarrows \Ccat_{\euX_{0}} \euY_{0} \rightarrow \Ccat_{\euX} \euY$ 
\end{center}
is an effective groupoid in the $\infty$-topos $\stacks$. Thus, $\Ccat_{\euX} \euY$ admits a smooth relatively Artin surjection from an Artin stack, and it follows that it itself must be Artin, see \cite{antieau_gepner_brauer_groups}[4.30].

We have a natural transformation $\Ccat \rightarrow \Ncat$ defined on the category of pairs of schemes, and since both the source and target are cosheaves, it follows from another application of \cref{Proposition:restriction_of_sheaves_to_generating_subcategory_is_an_equivalence} that this natural transformation uniquely extends to one defined on all relative Artin stacks. 

Since this natural transformation yields cartesian squares when applied to any smooth morphism of pairs of schemes, as we verified in \cref{Proposition:any_smooth_morphism_of_pairs_is_good_so_that_it_induces_cartesian_diagram_between_normal_cone_and_sheaf}, it follows formally through \cref{Proposition:cartesian_natural_transformations_of_adapted_cosheaves_are_stable_under_extension} that it has this property for any smooth morphism of relative Artin stacks. It follows from this that $\Ccat$ preserves smooth and smoothly surjective morphisms, finishing the proof of \cref{Theorem:there_exists_a_unique_normal_cone_functor_subject_to_properties}.

Since we already constructed the natural transformation $\Ccat \rightarrow \Ncat$ and we checked that it is smooth-cartesian, to prove \cref{Theorem:normal_cone_is_a_closed_substack_of_normal_sheaf_and_has_cartesian_property_for_smooth_morphisms} we're only left with checking that for any relative Artin stack $\euX \rightarrow \euY$, the resulting morphism $\Ccat_{\euX} \euY \rightarrow \Ncat _{\euX} \euY$ is a closed embedding. Choose a smooth surjection $(X \hookrightarrow Y) \rightarrow (\euX \rightarrow \euY)$ whose source is a pair of schemes, it follows that the diagram 

 \begin{center}
	\begin{tikzpicture}
		\node (TL) at (0, 1.5) {$ \Ccat_{X} Y $};
		\node (TR) at (1.5, 1.5) {$ \Ncat_{X} Y $};
		\node (BL) at (0 , 0) {$ \Ccat_{\euX} \euY $};
		\node (BR) at (1.5, 0){$ \Ncat_{\euX} \euY $};
		
		\draw[->] (TL) -- (TR);
		\draw[->] (BL) -- (BR);
		\draw[->] (TL) -- (BL);
		\draw[->] (TR) -- (BR);
	\end{tikzpicture}
\end{center}
is cartesian and that both vertical arrows are smooth surjections. Since it clear that the top horizontal arrow is a closed embedding, we deduce that the same is true for the bottom one, ending the proof.
\end{proof}

\section{The deformation space}
\label{Section:the_deformation_space}

Deforming a  closed embedding of schemes $ X \hookrightarrow Y$ into the zero section imbedding of $X$ into $\Ccat_X Y$ is a fundamental procedure in intersection theory, known as the deformation to the normal cone. In this section we will generalize this construction to any locally of finite type morphism of Artin stacks.

Recall that for a closed embedding $ X \hookrightarrow Y$ of schemes, the deformation $M^\circ_X Y$ is a flat scheme over $\mathbb{P}^{1}$ which fits into a commutative diagram

 \begin{center}
	\begin{tikzpicture}
		\node (TL) at (-1 , 1) {$ X \times \mathbb{P}^{1} $};
		\node (TR) at (1, 1) {$ M^\circ_X Y  $};
		\node (B) at (0, 0){$ \mathbb{P}^{1} $};
		
		\draw[right hook->] (TL) -- (TR);
		\draw[->] (TL) -- (B);
		\draw[->] (TR) -- (B);
	\end{tikzpicture}
\end{center}
such that

\begin{enumerate}
\item over $\mathbb{A}^{1} \simeq \mathbb{P}^{1} - \{ \infty \}$ the horizontal arrow is isomorphic to $X \times \mathbb{A}^{1} \hookrightarrow Y \times \mathbb{A}^{1}$ and
\item over $\{ \infty \}$, the horizontal arrow is isomorphic to $X \hookrightarrow \Ccat_{X}Y$.
\end{enumerate}

Explicitly, $ M^\circ_X Y$ can be constructed as the difference 

\begin{center}
$ M^\circ_X Y : =  \text{Bl}_{X \times \{ \infty \}} Y \times \mathbb{P}^1 - \text{Bl}_{X \times \{\infty \}} Y \times \{\infty \}$
\end{center}
between two blow-ups along $X \times \{ \infty \}$. Alternatively, if $X \hookrightarrow Y \simeq \spec(A)$ is defined by ideal $I$, then the restriction of $M^\circ_X Y $ to $\mathbb{P}^1 - \{ 0 \} \simeq \mathbb{A}^{1} \simeq \spec(k[t])$ is isomorphic to the spectrum of the Rees algebra $R(A, I) := \bigoplus_{k \in \mathbb{Z}} I^{k} t^{-k} \subseteq A[t, t^{-1}]$.

\begin{Lemma} 
\label{Lemma:the_deformation_space_sends_smooth_morphisms_of_pairs_to_smooth_morphisms_and_preserves_pullbacks}
The deformation space functor $M^\circ: \pairs \rightarrow \artinstacks _{/ \mathbb{P}^{1}}$ preserves coproducts, smooth morphisms, smooth surjections and commutes with pullbacks along smooth maps. In particular, it is a cosheaf adapted to the class of smooth maps. 
\end{Lemma}

\begin{proof} 
Since for any pair $X \hookrightarrow Y$ the deformation space $M^\circ_{X} Y$ is flat over $\mathbb{P}^{1}$, it suffices to check all of the claims fibrewise. This is clear, since $M^\circ_{X} Y \times _{\mathbb{P}^{1}} \{ t \} \simeq Y$ for $t \neq \infty$, $M^\circ_{X} Y \times _{\mathbb{P}^{1}} \{ \infty \} \simeq \Ccat_{X}Y$ and both of the right hand sides have these properties, the latter by \cref{Theorem:there_exists_a_unique_normal_cone_functor_subject_to_properties}.
\end{proof}

\begin{Theorem} 
\label{Theorem:existence_and_properties_of_deformation_space}
For any relative Artin stack $\euX \rightarrow \euY$ there exists an Artin stack $M^\circ_{\euX} \euY$ which fits into a commutative diagram

 \begin{center}
	\begin{tikzpicture}
		\node (TL) at (-1 , 1) {$ \euX \times \mathbb{P}^{1} $};
		\node (TR) at (1, 1) {$ M^\circ_{\euX} \euY  $};
		\node (B) at (0, 0){$ \mathbb{P}^{1} $};
		
		\draw[right hook->] (TL) -- (TR);
		\draw[->] (TL) -- (B);
		\draw[->] (TR) -- (B);
	\end{tikzpicture}
\end{center}
where both vertical arrows are flat and such that

\begin{enumerate}
\item over $\mathbb{A}^{1} \simeq \mathbb{P}^{1} - \{ \infty \}$, the horizontal arrow is isomorphic to $\euX \times \mathbb{A}^{1} \hookrightarrow \euY \times \mathbb{A}^{1}$ and
\item over $\{ \infty \}$, the horizontal arrow is isomorphic to $\euX \hookrightarrow \Ccat_{\euX}\euY$.
\end{enumerate}
\end{Theorem}

\begin{proof}
Since $M^{\circ}$ is an adapted cosheaf on the site of pairs by \cref{Lemma:the_deformation_space_sends_smooth_morphisms_of_pairs_to_smooth_morphisms_and_preserves_pullbacks}, it extends uniquely to an adapted cosheaf on all of $\relativeartinstacks$ by \cref{Theorem:extension_of_an_adapted_cosheaf_is_adapted}. It is easy to see that the formula $\euX \mapsto \euX \times \mathbb{P}^{1}$ also yields an adapted cosheaf, and so the natural transformation between the two defined for pairs also extends uniquely. 

To see that $M^\circ_{\euX} \euY \rightarrow \mathbb{P}^{1}$ is flat, choose a smooth surjection $(X \hookrightarrow Y) \rightarrow (\euX \rightarrow \euY)$ of relative Artin stacks whose source is a closed immersion of schemes. It then follows from \cref{Lemma:the_deformation_space_sends_smooth_morphisms_of_pairs_to_smooth_morphisms_and_preserves_pullbacks}, that $M^{\circ} _{X} Y \rightarrow M^\circ_{\euX} \euY$ is a smooth surjection, and since the composite $M^{\circ} _{X} Y \rightarrow \mathbb{P}^{1}$ is flat, we deduce the same is true for $M^\circ_{\euX} \euY \rightarrow \mathbb{P}^{1}$.

To deduce the two properties, observe that $\euY \times \mathbb{A}^{1} \rightarrow M^\circ_{\euX} \euY _{| \mathbb{A}^{1}}$ and $\Ccat_{\euX} \euY \rightarrow M^\circ_{\euX} \euY _{| \{ \infty \}}$ are natural transformations of adapted cosheaves on $\relativeartinstacks$ which restrict to isomorphisms for closed embeddings of schemes, and so must be equivalences in general. 
\end{proof}
The existence of the deformation space has the following important consequence, which in practice allows one to deduce many properties of the normal cone automatically. 

\begin{Corollary}
\label{Corollary:properties_stable_under_flat_deformation_are_the_same_for_normal_cone_and_target}
Let $P$ be a property of Artin stacks which is stable under flat deformation over an affine base. Then, for any relative Artin stack $\euX \rightarrow \euY$, $\euY$ has property $P$ if and only if the normal cone $\Ccat_{\euX} \euY$ has property $P$. 
\end{Corollary}

\begin{proof}
This is immediate from \cref{Theorem:existence_and_properties_of_deformation_space}.
\end{proof}

\section{The virtual fundamental class} 

In this chapter we introduce the notion of a perfect obstruction theory generalizing the classical one due to Behrend and Fantechi, and show that obstruction theories correspond to closed immersions under the the abelian cone functor. We then specialize to the case of an $1$-Artin stacks, where we have access to Chow groups, and construct the virtual fundamental class in the presence of global resolutions. Finally, we give a few examples of moduli stacks to which these methods apply.

\begin{Definition} 
\label{Definition:perfect_obstruction_theory} 
Let $\euX \rightarrow \euY$ be relative Artin stack and $ \varphi : \eupsilon \rightarrow L_{\euX/ \euY}[-1]$ be a morphism in $\qcoh(\euX)$. We say that $\varphi$ is an \emph{obstruction theory} if

\begin{enumerate}
\item The homomorphism $h_0(\varphi) $ is surjective
\item The homomorphism $ h_{i} (\varphi) $ is an isomorphism for $ i \leq -1$. 
\end{enumerate}
We say that an obstruction theory is \emph{perfect} if $\eupsilon$ is perfect of non-positive amplitude.
\end{Definition}
Keeping in mind that we use the homological grading convention, it is easy to see that if $\euX$ is Deligne-Mumford, our definition coincides up to a shift with the one given by Behrend and Fantechi in \cite{behrend_fantechi_intrinsic_normal_cone}. In this case, $L_{\euX}$ will be in fact connective. 

Informally, a perfect obstruction theory can be thought of as a "shadow" of a quasi-smooth derived enhancement, see \cref{Example:quasi-smooth_derived_stacks} for more detail. 

The abelian cone functor of \cref{Definition:abelian_cone_associated_to_a_qcoh_sheaf} provides us with a bridge from algebraic objects, namely quasi-coherent sheaves, to objects of geometric nature, namely abelian Artin stacks over $\euX$. We will now show what the condition of being an obstruction theory translates to in geometry, generalizing Behrend and Fantechi's criterion in the Deligne-Mumford case. 

\begin{Proposition}
\label{Proposition:criterion_for_a_morphism_of_qcoh_sheaves_to_get_affine_morphism_or_closed_immersion_on_cones}
Let $ \euX$ be an Artin stack and let $ \varphi : \eupsilon \rightarrow \EuScript{F}$ be a morphism of bounded below quasi-coherent sheaves. Then, $\mathbb{V}_{\euX}(\EuScript{F}) \rightarrow \mathbb{V}_{\euX}(\eupsilon)$ is
\begin{enumerate}
\item affine if and only if $ h_{-1}(\varphi) $ is surjective and $ h_{i} (\varphi)$ is an isomorphism for $ i \leq -2$ and 
\item a closed immersion if and only if $ h_{0}(\varphi)$ is surjective and $ h_{i} (\varphi)$ is an isomorphism for $ i \leq -1$.
\end{enumerate}
\end{Proposition}

\begin{proof}
As formation of the abelian cone commutes with arbitrary base-change by \cref{Lemma:base_change_for_cone_functor}, we can assume that $\euX \simeq \spec(A)$ is affine by replacing it by a smooth atlas. Before proceeding, let us observe that both of the homological conditions can be rephrased as saying that the cofibre of $\eupsilon \rightarrow \EuScript{F}$ is respectively, $0$- and $1$-connective. 

Since both the above homological conditions and the abelian cone only depend on the coconnective truncations, the latter as a consequence of \cref{Lemma:cone_of_qcoh_sheaf_only_depends_on_the_coconnective_part}, we can assume that $\eupsilon$ and $\EuScript{F}$ are coconnective. In this case, $\eupsilon$ can be represented by a non-positively graded chain complex 

\begin{center}
$0 \rightarrow E_{0} \rightarrow E_{-1} \rightarrow \ldots$
\end{center}
of $A$-modules, and since $\eupsilon$ is bounded below we can assume that $E_{i}$ are free for $i < 0$ and eventually vanish.

 In this case, we see from filtering $\eupsilon$ using the truncations of the given chain complex that for any $A$-module $M$, the induced map $\Ext^{0}_{A}(E_{0}, M) \rightarrow \Ext^{0}_{A}(\eupsilon, M)$ is surjective. Thus, the morphism $\mathbb{V}_{\spec(A)}(E_{0}) \rightarrow \mathbb{V}_{\spec(A)}(\eupsilon)$ is a surjection of stacks. It follows that the given map between abelian cones is affine or a closed immersion if and only if this is true for the base-change $\mathbb{V}_{\spec(A)}(E_{0}) \times _{\mathbb{V}_{\spec(A)}(\eupsilon)} \mathbb{V}_{\spec(A)}(\EuScript{F}) \rightarrow \mathbb{V}_{\spec(A)}(E_{0})$.
 
 Since the abelian cone takes colimits to limits, the above base-change can be identified with the map induced by $E_{0} \rightarrow \EuScript{F} \oplus _{\eupsilon} E_{0}$. Since this map has the same cofibre as $\eupsilon \rightarrow \EuScript{F}$, we see that by replacing $\eupsilon$ by $E_{0}$ and $\EuScript{F}$ by the pushout, we can assume that $\eupsilon$ is an $A$-module.
 
 If $\eupsilon$ is discrete, then $\mathbb{V}_{\spec(A)}(\eupsilon)$ is affine and we see that the morphism between cones is affine if only if $\mathbb{V}_{\spec(A)}(\EuScript{F})$ is affine. As a consequence of \cref{Lemma:abelian_cone_is_affine_if_and_only_qcoh_sheaf_connective}, this happens precisely when $\EuScript{F}$ is connective, which is equivalent to the first homological condition, as $h_{i}(\eupsilon) = 0$ for $i < 0$. 
 
To see that the second homological condition controls whether the morphism is a closed immersion, observe that in the case above when both $\EuScript{E}$ and $\EuScript{F}$ are $0$-connective, the map between cones can be identified with $\spec(\sym_{A}(h_{0}(\EuScript{F}))) \rightarrow \spec(\sym_{A}(h_{0}(\eupsilon)))$. This is clearly a closed immersion if and only if $h_{0}(\eupsilon) \rightarrow h_{0}(\EuScript{F})$ is a surjection, ending the proof. 
\end{proof}

\begin{Corollary} 
\label{Corollary:perfect_obstruction_geometrizion} 
Let $\euX \rightarrow \euY$ be a relative Artin stack. Then, $\varphi : \eupsilon \rightarrow L_{\euX/\euY}[-1]$ is an obstruction theory if and only if $\eupsilon$ is bounded below and $\Ncat_{\euX} \euY \rightarrow \mathbb{V}_\euX (\eupsilon)$ is a closed immersion. 
\end{Corollary}

\begin{proof}
This is immediate from \cref{Proposition:criterion_for_a_morphism_of_qcoh_sheaves_to_get_affine_morphism_or_closed_immersion_on_cones}.
\end{proof}
Recall that if $\euX \rightarrow \euY$ is a morphism of Artin stacks, then the cotangent complex $L_{\euX / \euY}$ controls the deformation theory in the sense that for any $A$-valued point $\eta: \spec(A) \rightarrow \euX$, an $A$-module $M$, and $\tilde{A}$ a square-zero extension of $A$ by $M$  and a diagram 

\begin{center}
\begin{tikzpicture}
		\node (BLANK) at (-6.5, 0) {$ $};
		\node (TL) at (0, 1.4) {$\spec(A)$};
		\node (TR) at (2.5, 1.4) {$\spec(\tilde{A})$};
		\node (BL) at (0, 0) {$\euX$};
		\node (BR) at (2.5, 0) {$\euY$};
		\node (SYMBOL) at (7.75, 0.7) {$ (\spadesuit) $};

		\draw [->] (TL) -- (TR);
		\draw [->] (BL) -- (BR);
		\draw [->] (TR) -- (BR);
		\draw [->] (TL) -- (BL) node [midway, left] {$\eta$};
		\draw[->] [dotted] (TR) -- (BL);
	\end{tikzpicture},
\end{center}
the extension denoted above by the dotted arrow exists if and only if the associated obstruction in $\Ext^{1}_{A}(\eta^{*} L_{\euX / \euY}, M)$ vanishes. More generally, the cotangent complex has this property also for square-zero extensions of derived rings, by which it is then determined uniquely, see the discussion proceeding \cref{Definition:cotangent_complex_of_a_morphism_of_stacks}.

This uniqueness does not hold if we consider only discrete rings, in fact we will now prove a minor generalization of a criterion of Behrend and Fantechi which tells us that a morphism $\eupsilon[1] \rightarrow L_{\euX / \euY}$ is a shift of an obstruction theory if and only if $\eupsilon[1]$ also controls the deformation theory of affine schemes mapping into $\euX$. 

\begin{Proposition}
\label{Proposition:a_morphism_is_an_obstruction_theory_iff_it_induces_iso_on_ext_in_a_range}
A morphism $\eupsilon[1] \rightarrow L_{\euX / \euY}$ is a shift of an obstruction theory if and only if for any point $\eta: \spec(A) \rightarrow \euX$ and any $A$-module $M$, the induced morphism 

\begin{center}
$\Ext_{A}^{i}(\eta^{*} L_{\euX / \euY}, M) \rightarrow \Ext_{A}^{i}(\eta^{*} \eupsilon[1], M)$
\end{center}
is injective for $i = 1$ and an isomorphism for $i \leq 0$.
\end{Proposition}

\begin{proof}
Since a morphism $\eupsilon[1] \rightarrow L_{\euX / \euY}$ is a shift of an obstruction theory if and only if its cofibre is $2$-connective, the statement is equivalent to saying that $\EuScript{C} \in \qcoh(\euX)$ is $2$-connective if and only if $\Ext^{k}(\eta^{*} \EuScript{C}, M) = 0$ for any $\eta$, $M$ as above and $k \leq 1$. 

This is clear, since $\EuScript{C}$ is $2$-connective if and only if $\eta^{*}\EuScript{C}$ is $2$-connective for all $\eta$, and that's equivalent to saying that $\Ext^{k}(\eta^{*} \EuScript{C}, \EuScript{M}) = 0$ for any $k \geq 1$ and $\EuScript{M} \in \qcoh(\spec(A))_{\leq 0}$. Since the latter $\infty$-category is generated under limits by $A$-modules, it is enough to check this condition in this case, ending the argument.
\end{proof}

\begin{Corollary}
\label{Corollary:a_morphism_is_an_obstruction_theory_if_and_only_if_it_obstructs_square_zero_extensions}
Let $\phi: \eupsilon[1] \rightarrow L_{\euX / \euY}$ be a morphism of quasi-coherent sheaves. Then, $\phi$ is a shift of an obstruction theory if and only if for any diagram of the form $(\spadesuit)$

\begin{enumerate}
\item the dotted arrow exists if and only if the associated obstruction in $\Ext^{1}(\eta^{*} \eupsilon, M)$ vanishes, and if this is the case then 
\item the space of such dotted arrows is equivalent to $\map_{\qcoh(A)}(\eta^{*} \eupsilon, M)$, in particular their homotopy classes form an $\Ext^{0}_{A}(\eta^{*} \eupsilon, M)$-torsor.
\end{enumerate}
\end{Corollary}

\begin{proof}
It is clear that $L_{\euX / \euY}$ has this property, and the statement is then an immediate consequence of \cref{Proposition:a_morphism_is_an_obstruction_theory_iff_it_induces_iso_on_ext_in_a_range}.
\end{proof}

We now move on to the construction of the virtual fundamental class. Let us now restrict to the case where $\euX \rightarrow \euY$ is a morphism of finite type $1$-Artin stacks, where we have access to  the Chow groups as constructed by Kresch \cite{Kr}. 

\begin{Definition}
\label{Definition:global_resolution}
If $\eupsilon$ is a perfect obstruction theory, then a \emph{global resolution} is a morphism $\eupsilon \rightarrow E$ injective on $h_{0}$ such that $E \in \qcoh(\euX)^{\heartsuit}$ is a locally free sheaf of finite rank. 
\end{Definition}

\begin{Construction}
\label{Construction:closed_substack_of_global_resolution_associated_to_normal_cone}
Suppose that we have a morphism of $1$-Artin stacks $\euX \rightarrow \euY$ with target purely of dimension $r$, in which case the same is true for the normal cone $\Ccat_{\euX} \euY$ as a consequence of \cref{Corollary:properties_stable_under_flat_deformation_are_the_same_for_normal_cone_and_target}.

If $\eupsilon \rightarrow E$ is a global resolution, then \cref{Lemma:abelian_cone_takes_certain_maps_of_qcoh_sheaves_to_atlases} implies that $\mathbb{V}_\euX (E) \rightarrow \mathbb{V}_\euX (\eupsilon)$ is a smooth surjection, and since the source is a vector bundle it forms a smooth atlas for $\mathbb{V}_\euX (\eupsilon)$ relative to $\euX$. We can then consider the pullback diagram
\begin{center}
	\begin{tikzpicture}
		\node (TL) at (0, 1.3) {$ \mathbb{V}_\euX (E) \times _{\mathbb{V}_{\euX}(\eupsilon)} \Ccat _{\euX} \euY$};
		\node (TR) at (3, 1.3) {$\mathbb{V}_\euX (E)$};
		\node (BL) at (0, 0) {$\Ccat_\euX \euY$};
		\node (BR) at (3, 0) {$\mathbb{V}_\euX(\eupsilon).$};
		
		\draw [->] (TL) -- (TR);
		\draw [->] (TL) -- (BL);
		\draw [->] (TR) -- (BR);
		\draw [->] (BL) -- (BR);
	\end{tikzpicture},
\end{center}
where the bottom map is the composite  $\Ccat _{\euX} \euY \hookrightarrow \Ncat_{\euX} \euY \simeq \mathbb{V}_{\euX}(L_{\euX / \euY}[-1]) \hookrightarrow \mathbb{V}_{\euX}(\eupsilon)$, which is a closed embedding as a consequence of \cref{Corollary:perfect_obstruction_geometrizion}. 
\end{Construction}
In the setting of \cref{Construction:closed_substack_of_global_resolution_associated_to_normal_cone}, we can now define the virtual fundamental class.
\begin{Definition}
\label{Definition:virtual_fundamental_class}
Let $\euX \rightarrow \euY$ be a morphism of $1$-Artin stacks as above. Then, the \emph{virtual fundamental class} associated to a perfect obstruction theory $\eupsilon \rightarrow L_{\euX / \euY}[-1]$ which admits a global resolution $E$ is given by

\begin{center}
$[\euX \rightarrow \euY, \eupsilon]^{vir} : = 0^! [\mathbb{V}_\euX (E) \times _{\mathbb{V}_{\euX}(\eupsilon)} \Ccat _{\euX} \euY] \in \chow_{r - \chi(\eupsilon)}(\euX)$
\end{center}
where $0: \euX \rightarrow C_\euX (E)$ is the zero section. 
\end{Definition}
A priori our construction of the virtual class depends on the choice of global resolution $E$ we used to define it, we will now show that it is in fact canonically attached to the perfect obstruction theory $\eupsilon$. 

\begin{Proposition}
\label{Proposition:invariance_of_global_resolution}
The virtual fundamental class $[\euX \rightarrow \euY, \eupsilon]^{vir}$ is independent of the choice of a global resolution of a perfect obstruction theory $\eupsilon$.
\end{Proposition}

\begin{proof}
If $\eupsilon \rightarrow E$, $\eupsilon \rightarrow F$ are global resolutions, then it is easy to see that the same is true for $\eupsilon \rightarrow E \oplus F$. Thus, it is enough to check that the virtual fundamental class constructed using $E$ coincides with that of $E \oplus F$. We have a commutative diagram 

\begin{center}
	\begin{tikzpicture}
		\node (TTL) at (0, 2.6) {$ \mathbb{V}_\euX (E \oplus F) \times _{\mathbb{V}_{\euX}(\eupsilon)} \Ccat _{\euX} \euY $};
		\node (TTR) at (3.5, 2.6) {$ \mathbb{V}_\euX (E \oplus F) $};
		\node (TL) at (0, 1.3) {$ \mathbb{V}_\euX (E) \times _{\mathbb{V}_{\euX}(\eupsilon)} \Ccat _{\euX} \euY$};
		\node (TR) at (3.5, 1.3) {$\mathbb{V}_\euX (E )$};
		
		\draw [->] (TTL) -- (TTR);
		\draw [->] (TTL) -- (TL);
		\draw [->] (TTR) to node[auto] {$ p $} (TR);
		\draw [->] (TL) -- (TR);
	\end{tikzpicture},
\end{center}
and so $p^{*} [\mathbb{V}_\euX (E) \times _{\mathbb{V}_{\euX}(\eupsilon)}  \Ccat _{\euX} \euY ] = [\mathbb{V}_\euX (E \oplus F) \times _{\mathbb{V}_{\euX}(\eupsilon)} \Ccat _{\euX} \euY]$ as elements of $\chow(\mathbb{V}_\euX (E \oplus F))$. Then, the needed equality is obtained by intersecting with the zero sections of $E$ and $E \oplus F$, since $\pi_{E \oplus F} \simeq p \circ \pi_{E}$ implies $0_{E \oplus F}^{!} \circ p^{*} \simeq 0_{E}^{!}$. 

\end{proof}

\begin{Remark}
\label{Remark:virtual_fundamental_class_in_k_theory}
Note that the only reason we restricted to $1$-Artin stacks is that we needed a suitably well-behaved theory of Chow groups. It is clear that the above formula gives a fundamental class associated to any suitable homology theory of relative Artin stacks. In particular,

\begin{center}
$0^{*} [ \EuScript{O} _{\Ccat_{\euX}}] \in K_0 (\coh(\euX))$,
\end{center}
where $0: \euX \rightarrow \mathbb{V}_{\euX}(\eupsilon)$ is the zero section, defines a \emph{virtual fundamental class in $G$-theory} of a finite type Artin stack $\euX$ equipped with a choice of a perfect obstruction theory.
\end{Remark}

We now give a few examples of applications of our constructions.

\begin{Example}[Intersection theory]
\label{Example:the_basic_example}
Suppose we have a cartesian diagram

\begin{center}
 	\begin{tikzpicture}
		\node (TL) at (0, 1.3) {$\euW$};
		\node (TR) at (2, 1.3) {$\euX$};
		\node (BL) at (0, 0) {$\euY$};
		\node (BR) at (2, 0) {$\euZ$};
		
		\draw [->] (TL) -- (TR) node [midway, above] {$j$};
		\draw [->] (TL) -- (BL) node [midway, left] {$g$};
		\draw [->] (TR) -- (BR) node [midway, right] {$f$};
		\draw [->] (BL) -- (BR) node [midway, below] {$i$};
	\end{tikzpicture}
\end{center}
of $1$-Artin stacks such that $ \euX$ and $\euZ$ are smooth, $\euX$ has the resolution property and $i$ is a regular closed embedding. Consider the cofibre sequence

\begin{center}
$g^*L_{\euY/\euZ}[-1] \rightarrow j^*L_\euX \rightarrow E$.
\end{center}
where the left map is induced by the morphisms $ L_{\euY/\euZ}[-1] \rightarrow i^*L_{\euY}$ and $ f^*L_{\euZ} \rightarrow L_{\euX}$. Since $\euX$ is smooth, $L_{\euX}$ has perfect amplitude in $[0,-1]$ and since $i$ is regular, $ L_{\euY/\euZ}$ is equivalent to $\EuScript{I}/\EuScript{I}^2 [1]$. It follows that $E$ is perfect, and one easily observes that the induced morphism

\begin{center}
$E \rightarrow L_\euW$
\end{center}
is in fact a perfect obstruction theory. We deduce that $\euW$ admits a virtual fundamental class.

In the case of schemes, the resulting class coincides with $i^{!}[\euX]$ in the classical sense, as observed by Behrend and Fantechi \cite{behrend_fantechi_intrinsic_normal_cone}[6.1]. Thus, the above construction can be thought of as generalizing Fulton's construction to the setting of Artin stacks, recovering Kresch's Gysin maps. 
\end{Example}

\begin{Example}[Twisted stable maps]
\label{Example:twisted_stable_maps}
Let $\euX$ be a finitely presented, proper, smooth, tame $1$-Artin stack with finite inertia. Moreover, suppose that $\euX$ has the resolution property, that the coarse moduli space of $\euX$ is projective, and that we have fixed an element $ \beta \in \chow^{\text{num}}_1(X)$. 

In this context, one can show that the canonical morphism 
\begin{center}
$\EuScript{K}_{g, n} (\EuScript{X}, \beta) \rightarrow \twistedcurves_{g,n}$
\end{center}
from the moduli stack of twisted stable maps to the moduli stack of twisted curves has a perfect obstruction theory which admits a global resolution. To do so, one considers the diagram 
\begin{center}
	\begin{tikzpicture}
		\node (TL) at (0, 1.3) {$\mathscr{C}$};
		\node (TM) at (3, 1.3) {$\overline{\mathscr{C}}$};
		\node (TR) at (5, 1.3){$\euX$};
		\node (BL) at (0 , 0) {$\EuScript{K}_{g,n} (\euX, \beta)$};
		\node (BM) at (3, 0){$\underline{\hom}_{\twistedcurves_{g,n}} (\EuScript{C}, \euX)$};
		\node(BBM) at (3,-1.3) {$\twistedcurves_{g,n}$};
		
		\draw[->] (TL) -- (TM);
		\draw[->] (TM)--(TR) node [midway, above] {$\bar{\psi}$};
		\draw[->] (BL)--(BM) node [midway, above] {$\iota$} ;
		\draw[->] (TL) -- (BL) node [midway, left] {$\pi$};
		\draw[->] (TM)--(BM) node [midway, right] {$\bar{\pi}$};
		\draw[->] (BM) -- (BBM) node [midway, right] {$\bar{\varphi}$};
		\draw [->] (BL) -- (BBM) node [midway, below] {$\varphi$};
	\end{tikzpicture}
\end{center}
where $\EuScript{C}$ is the universal twisted curve and $\overline{\mathscr{C}} : = \EuScript{C} \times_{\twistedcurves_{g,n}} \underline{\hom}_{\twistedcurves_{g,n}} (\EuScript{C}, \euX)$. Then, by Grothendieck duality there is a canonical morphism 
\begin{center}
$ \bar{\pi}_*(\bar{\psi}^* L_\euX \otimes \omega_{\bar{\pi}})|_{\EuScript{K}_{g,n}(\euX, \beta)} \rightarrow L_\varphi [-1]$
\end{center}
which is shown to be a perfect obstruction theory using \cref{Corollary:a_morphism_is_an_obstruction_theory_if_and_only_if_it_obstructs_square_zero_extensions}. Moreover, this obstruction theory has a global resolution due to the fact that $\EuScript{X}$ has the resolution property and we deduce that the stack $\EuScript{K}_{g, n} (\EuScript{X}, \beta)$ admits a relative virtual fundamental class. In future work, will prove that this class satisfies the Gromov-Witten axioms.
\end{Example}

\begin{Example}[Quantum $K$-theory]
In the context of the previous example, one can instead consider the fundamental class in $K$-theory discussed in \cref{Remark:virtual_fundamental_class_in_k_theory}.
This class is related to quantum K-theory in the sense of Lee \cite{lee2004quantum}, which we hope to revisit in future work.
\end{Example}

\begin{Example}[Quasi-smooth derived stacks]
\label{Example:quasi-smooth_derived_stacks} 
Let $\euX \rightarrow \euY$ be a quasi-smooth morphism of derived $1$-stacks \cite{Higherandderived}, \cite{khan2020algebraic}. If $\iota: \euX^{cl} \hookrightarrow \euX $ denotes the inclusion of the underlying classical stack, then the canonical morphism 

\begin{center}
$i^*L_{\euX/ \euY} [-1] \rightarrow L_{\euX^{cl} / \euY^{cl}} [-1]$
\end{center} 
can be shown to be a perfect obstruction theory on $\euX^{cl}$ using the connectivity estimates given in \cite{spectral_algebraic_geometry}[I.1.2.5.6]. If $i^*L_{\euX/ \euY} [-1]$ has a global resolution, which is always the case if $\euX^{cl}$ has the resolution property, it follows that we have a virtual fundamental class $[\euX^{cl} \rightarrow \euY^{cl}, i^*L_{\euX/ \euY}]^{vir}$. 
\end{Example}
 
\bibliographystyle{amsalpha}
\bibliography{intrinsic_normal_cone_bibliography}

\end{document}